\documentclass{amsart}

\usepackage[utf8]{inputenc}
\usepackage{amsmath}
\usepackage{amssymb}
\usepackage{url}
\usepackage{enumerate}
\usepackage{hyperref}

\newtheorem{definition}{Definition}
\newtheorem{lemma}[definition]{Lemma}
\newtheorem{proposition}[definition]{Proposition}
\newtheorem{theorem}[definition]{Theorem}
\newtheorem{conjecture}[definition]{Conjecture}
\newtheorem{corollary}[definition]{Corollary}
\newtheorem{algorithm}[definition]{Algorithm}
\theoremstyle{remark}
\newtheorem{remark}[definition]{Remark}

\newcounter{enumAlgoC}
\newenvironment{enumAlgo}{\begin{list}{\footnotesize \arabic{enumAlgoC}}{
\usecounter{enumAlgoC}
\setlength{\itemsep}{.3em}
\setlength{\topsep}{.3em}
\setlength{\leftmargin}{1.5em}
\setlength{\labelwidth}{1.5em}
\setlength\labelsep{.7em}
}}{\end{list}}
\newlength{\algleftskip}
\newlength{\algboxwidth}
\newcommand{\algbox}[2]{%
\vskip.3em\par\noindent\hspace{\algleftskip}\parbox[t]{\algboxwidth}{\hskip-\algleftskip\textit{#1:\hspace{1ex}}#2\vskip.2em}\par}
\newcommand{\Input}[1]{\algbox{Input}{#1}}
\newcommand{\Output}[1]{\algbox{Output}{#1}}
\newcommand{\Return}[1]{\vskip-.3em\algbox{Return}{#1}}

\newcommand{\algpart}[1]{\item[]\vskip.5em\hspace*{-1.5em}\textit{(#1)}}

\usepackage{tikz}
\usetikzlibrary{matrix,arrows}
\tikzstyle default=[baseline=-.3em]
\tikzstyle mtrx=[matrix of math nodes,row sep=1.2em,column sep=1.5em, text height=1.5ex, text depth=0.25ex]
\tikzstyle sizeM=[row sep=1.5em,column sep=2em]
\tikzstyle sizeL=[row sep=2em,column sep=2.5em]
\tikzstyle flds=[row sep=2em,column sep=2em]
\tikzstyle ar=[above, shift={(-0.05,-0.05)}]
\tikzstyle arr=[above, shift={(-0.05,-0.1)}]
\tikzstyle r=[right,shift={(-0.05,0)}]

\DeclareMathOperator{\coker}{coker}
\DeclareMathOperator{\Gl}{Gl}
\newcommand{\Gal}{\mathrm{Gal}}
\newcommand{\Ext}{\mathrm{Ext}}

\DeclareMathOperator{\im}{im}
\newcommand{\ind}{\mathrm{ind}}
\newcommand{\inv}{\mathrm{inv}}
\newcommand{\Irr}{\textrm{Irr}}
\DeclareMathOperator{\Nr}{N}
\DeclareMathOperator{\nr}{nr}
\DeclareMathOperator{\Quot}{Quot}
\newcommand{\tor}{\mathrm{tor}}
\DeclareMathOperator{\Zent}{Z}

\newcommand{\NN}{\mathbb{N}}
\newcommand{\ZZ}{\mathbb{Z}}
\newcommand{\QQ}{\mathbb{Q}}
\newcommand{\RR}{\mathbb{R}}
\newcommand{\CC}{\mathbb{C}}
\newcommand{\calD}{\mathcal{D}}
\newcommand{\calI}{\mathcal{I}}
\newcommand{\calL}{\mathcal{L}}
\newcommand{\calM}{\mathcal{M}}
\newcommand{\calO}{\mathcal{O}}
\newcommand{\fra}{\mathfrak{a}}
\newcommand{\frf}{\mathfrak{f}}
\newcommand{\frg}{\mathfrak{g}}

\newcommand{\frm}{\mathfrak{m}}
\newcommand{\frP}{\mathfrak{P}}
\newcommand{\frp}{\mathfrak{p}}
\newcommand{\frQ}{\mathfrak{Q}}
\newcommand{\frq}{\mathfrak{q}}
\usepackage{mathrsfs}
\newcommand{\scrE}{\mathscr{E}}
\newcommand{\scrL}{\mathscr{L}}

\newcommand{\riightarrow}{\hbox to 3ex{\rightarrowfill}}
\newcommand{\loongrightarrow}{\hbox to 5ex{\rightarrowfill}}
\newcommand{\looongrightarrow}{\hbox to 6ex{\rightarrowfill}}
\newcommand{\vlongrightarrow}[1]{\hbox to #1{\rightarrowfill}}
\newcommand{\hooklongrightarrow}{\lhook\joinrel\longrightarrow}
\newcommand{\simeqarrow}{\stackrel{\!\simeq}\riightarrow}
\newcommand{\simeqlongarrow}{\stackrel{\!\simeq}\longrightarrow}

\newcommand{\TateH}{\hat H}
\newcommand{\EPS}{{\rm EPS}}
\newcommand{\EPSloc}{\EPS^{\rm loc}}
\newcommand{\frob}{\varphi}

\newcommand{\wtilde}[1]{\hspace*{.5ex}\widetilde{\hspace{-.5ex}#1}}
\newcommand{\what}[1]{\hspace*{.3ex}\widehat{\hspace{-.3ex}#1}}
\newsavebox{\tildebox}%
\newlength{\tildelen}%
\newcommand{\mytilde}[1]{%
\sbox{\tildebox}{$#1$}%
\settowidth{\tildelen}{\usebox{\tildebox}}%
\raisebox{2.25pt}{$\wtilde{\hspace{\tildelen}}$}%
\hspace{-\tildelen}\hspace{0pt}\usebox{\tildebox}}

\renewcommand{\eqref}[1]{\hyperref[#1]{(\ref*{#1})}}
\usepackage{ifthen}
\def\subStr[#1:#2]{#1}
\def\refPrefix#1{\expandafter\subStr[#1]}
\renewcommand{\autoref}[1]{%
\hyperref[#1]{%
\edef\pref{\refPrefix{#1}}%
\ifthenelse{\equal{\pref}{prop}}{Proposition~}{%
\ifthenelse{\equal{\pref}{lem}}{Lemma~}{%
\ifthenelse{\equal{\pref}{thm}}{Theorem~}{%
\ifthenelse{\equal{\pref}{alg}}{Algorithm~}{%
\ifthenelse{\equal{\pref}{rem}}{Remark~}{%
\ifthenelse{\equal{\pref}{sec}}{§}{%
\ifthenelse{\equal{\pref}{cor}}{Corollary~}{%
\ifthenelse{\equal{\pref}{conj}}{Conjecture~}{%
}}}}}}}}\ref*{#1}}}
\begin{document}

\title{Algorithmic Proof of the Epsilon Constant Conjecture}

\author[W. Bley]{Werner Bley}
\address{Werner Bley\\
Universität München\\
Theresienstr. 39\\
80333 München\\
Germany}
\email{bley@math.lmu.de}

\author[R. Debeerst]{Ruben Debeerst}
\address{Ruben Debeerst\\
Universität Kassel\\
Heinrich-Plett-Str. 40\\
34132 Kassel\\
Germany}
\email{debeerst@math.uni-kassel.de}
\thanks{}


%

\date{Version of 23rd Feb 2012}


\thanks{The second author was supported by DFG grant BL 395/3-1}

\begin{abstract}
In this paper we will algorithmically prove
the global epsilon constant conjecture for all Galois extensions $L/\QQ$ of degree
at most $15$. In fact, we will obtain a slightly more general result whose proof is based
on an algorithmic proof of the local epsilon constant conjecture for Galois extensions
$E/\QQ_p$ of small degree. 
To this end we will present an efficient algorithm for the
computation of local fundamental classes and address several
other problems arising in the algorithmic proof of the local conjecture.
\end{abstract}

\maketitle

\section{Introduction}
\label{sec:into}
For a tamely ramified Galois extension $L/K$ of number fields
with Galois group $G$,
the ring of integers $\calO_L$ has been studied as a projective $\ZZ[G]$-module.
Cassou-Nogu\`{e}s and Fröhlich defined a root number class $W_{L/K}$ 
associated to the epsilon constants occuring in the functional equation of Artin $L$-functions,
and it was conjectured by Fröhlich and proved by Taylor in 1981
that this class is equal to the class of $\calO_L$ 
in the reduced projective class group ${\rm Cl}(\ZZ[G])$,
see~\cite{Froe:GalMod,Taylor}.

In 1985 Chinburg defined an element $\Omega(L/K,2)$ in ${\rm Cl}(\ZZ[G])$
for arbitrary Galois extensions $L/K$ with $\Gal(L/K) = G$
using cohomological methods and proved that it
matches the class  of $\calO_L$ for tamely ramified extensions.
His $\Omega(2)$-conjecture, stating the equality of $\Omega(L/K,2)$
and $W_{L/K}$ in ${\rm Cl}(\ZZ[G])$,
therefore generalizes Fröhlich's conjecture
to wildly ramified extensions, cf.~\cite{Chin:ExactSeq}.

Later, Burns and the first named  author formulated in \cite{BleyBur:Equiv} a conjectural description of
epsilon constants in the relative algebraic $K$-group $K_0(\ZZ[G],\RR)$,
which implies Chinburg's $\Omega(2)$-conjecture 
via the canonical surjection $K_0(\ZZ[G],\RR)\twoheadrightarrow {\rm Cl}(\ZZ[G])$.
More precisely, they define an equivariant epsilon constant  $\scrE_{L/K}$ in $K_0(\ZZ[G],\RR)$
and an element involving algebraic invariants,
which project to the root number class and to $\Omega(L/K,2)$, respectively.
If we denote the difference of $\scrE_{L/K}$ and the algebraic invariant by  $T\Omega^{\rm loc}(L/K,1)$, 
then the global epsilon constant conjecture predicts
the vanishing of $T\Omega^{\rm loc}(L/K,1)$ in $K_0(\ZZ[G],\RR)$.
Burns and the first named  author proved their conjecture for tamely ramified
extensions and for abelian extensions of $\QQ$ with odd conductor.
They also proved that $T\Omega^{\rm loc}(L/K,1)$
is an element of the torsion subgroup $K_0(\ZZ[G],\QQ)_\tor$ of $K_0(\ZZ[G],\RR)$,
see~\cite[Cor.\ 6.3]{BleyBur:Equiv}.
The results of Breuning and Burns in \cite{BreBur:Dedekind} imply that the
conjecture is true for all abelian extensions of $\QQ$.

The global epsilon constant conjecture fits into the more general framework of the
equivariant Tamagawa number conjecture (ETNC) formulated by Burns and Flach in \cite{BurFla}.
ETNC conjecturally describes the leading term of an  equivariant motivic $L$-function at integer arguments
by cohomological invariants.
In the number field case, the global epsilon constant conjecture of \cite{BleyBur:Equiv} 
is in fact motivated by the conjectured compatibility of the ETNC 
for the leading terms of  Artin $L$-functions at $s=0$ and $s=1$ with the functional equation.
In \cite{BreBur:07:LeadingTerms} Burns and Breuning study explicit variants of these two conjectures
and the relation to the global epsilon constant conjecture.
More recently, they proved in \cite{BreBur:Dedekind} assuming Leopoldt's conjecture
that their explicit conjecture
at $s=1$ is equivalent to the relevant case of the ETNC.

The decomposition $K_0(\ZZ[G],\QQ)=\bigoplus_p K_0(\ZZ_p[G],\QQ_p)$
splits 
$T\Omega^{\rm loc}(L/K,1)$ into $p$-parts.
This has been further refined by Breuning in \cite{Breu:LocEps}
stating an independent conjecture for Galois extensions $E/F$ of local number fields in
the group $K_0(\ZZ_p[D],\QQ_p)$ where $D = \Gal(E/F)$.
He defined an element $R_{E/F}\in K_0(\ZZ_p[D], \QQ_p)$
incorporating local epsilon constants and algebraic invariants
associated to the local number field extension $E/F$
and conjectured the vanishing of $R_{E/F}$.
Breuning proved his local epsilon constant conjecture in \cite{Breu:LocEps}
for tamely ramified extensions, for abelian extensions of $\QQ_p$
with $p\neq2$, for all $S_3$-extensions,
and for certain infinite families of dihedral and quaternion extensions.

Moreover, this local conjecture is related to the global conjecture by
the equation
$T\Omega^{\rm loc}(L/K,1)_p = \sum_v {\rm i}_{G_w}^G (R_{L_w/K_v})$
where $v$ runs through all places of $K$ above $p$, $w$ is a fixed place of $L$ above $v$,
$G_w$ denotes the decomposition group
and ${\rm i}_{G_w}^G$ is the induction map on 
the relative $K$-group, 
cf.\ \cite[Thm.\ 4.1]{Breu:LocEps}.

We fix a base field $K$ and a finite group $G$. 
Using the result for tame extensions,
one concludes that the validity of the global conjecture
for all Galois extensions $L/K$ with $\Gal(L/K)$ isomorphic to $G$ 
depends upon the validity
of the local conjecture for only finitely many local
extensions.

Subsequently, Breuning and the first named  author presented an 
algorithm in \cite{BleyBreu:ExactAlgo} which proves the local epsilon constant conjecture
for a given local number field extension.
To establish a practical algorithm, there were, however,
still some tasks which needed a more efficient solution.
In this paper we will address these problems, give solutions,
and present computational results.

These computations together with known theoretical results will prove:
\begin{theorem}
\label{thm:loceps}

a) If $p$ is odd, then the local epsilon constant conjecture is valid for all
 Galois extensions $E$ of $\QQ_p$ with $[E:\QQ_p]\le 15$.

b) If $p=2$, then the  local epsilon constant conjecture is valid for all
non-abelian Galois extensions $E$ of $\QQ_2$ with $[E:\QQ_2]\le 15$. In addition,
it is valid for all abelian extensions $E$ of $\QQ_2$ with $[E:\QQ_p]\le 7$.

\end{theorem}

\begin{remark}
  In the statement of Theorem \ref{thm:loceps} we only considered extensions $E/\QQ_p$ of degree
$\le 15$ because only in these cases our computations led to new results. As already mentioned, the
local epsilon constant conjecture is proven by Breuning for all tamely ramified extensions ($p$ arbitrary) and all abelian
extensions $E/\QQ_p$ with $p \ne 2$.

We point out that we could not prove the local conjecture for wildly ramified abelian extensions $E/\QQ_2$ of degree
$8 \le [E : \QQ_2] \le 15$. The main reason for this is that the unramified extension of degree 8 over $\QQ_2$
cannot be represented as the completion of a Galois extension $L/\QQ$ of degree 8.%
\footnote{This is Wang's counterexample to Grunwald's original statement of his theorem.}
Instead, we have to use an extension $L/K$ of degree 8, with $[L:\QQ]=16$ and $[K:\QQ]=2$,
and this increases the complexity of the computations a lot.
\end{remark}

The above relation between
$T\Omega^{\rm loc}(L/K,1)_p$ and $R_{L_w/K_v}$ will imply results for global Galois extensions $L/\QQ$
which satisfy the following property. 

\medskip
\noindent
{\bf Property $(*)$}
We say that the Galois extension $L/K$ of number fields satisfies Property $(*)$ if for every wildly
ramified place $v$ of $K$ with $w | v | p$ one of the following cases is satisfied
\begin{align*}
  a) \quad& K_v = \QQ_p, p>2 \text{ and } G_w \text{ is abelian},\\
  b) \quad& K_v = \QQ_p, p=2, G_w \text{ is abelian and } |G_w|\le 7\\
  c) \quad& K_v = \QQ_p, p\ge2, G_w \text{ is non-abelian and } |G_w|\le 15. \\[-.5em]
\end{align*}

\begin{corollary}
\label{cor:globeps}
The global epsilon constant conjecture is valid for all
Galois extensions $L/K$ which satisfy Property $(*)$.
\end{corollary}

The projection onto the class group also proves
Chinburg's conjecture:

\begin{corollary}
\label{cor:omega2}
Chinburg's $\Omega(2)$-conjecture is valid for all
Galois extensions $L/K$  which satisfy Property $(*)$.
\end{corollary}

Moreover, the functorial properties of \cite[Prop.\ 3.3]{Breu:LocEps}
imply the following result:

\begin{corollary}
\label{cor:relative}
The global epsilon constant conjecture and Chinburg's $\Omega(2)$-conjecture
are valid for global Galois extensions $E/F$
for which $K \subseteq F \subseteq E \subseteq L$ with a Galois extension $L/K$  that satisfies Property~$(*)$.
\end{corollary}

\begin{remark}\label{A5 remark}
If $L/\QQ$ is a Galois extension with $\Gal(L/\QQ) \simeq A_5$, the alternating group of order $60$, then
$L/\QQ$ satisfies Property $(*)$. More generally, if $K$ is a number field where $2,3$ and $5$ are completely split,
then each $A_5$-extension $L/K$ satisfies Property~$(*)$.
\end{remark}

\begin{corollary}
\label{cor:deg 15}
The global epsilon constant conjecture and Chinburg's $\Omega(2)$-conjecture are valid for
all global Galois extensions $E/F$ for which  $\QQ\subseteq F \subseteq E \subseteq L$
with a Galois extension $L/\QQ$ of degree $[L : \QQ] \le 15$. 
\end{corollary}

The main contents of this article are as follows.
After introducing some notation 
we will present an efficient algorithm for the
computation of local fundamental classes in \autoref{sec:lfc}.
In \autoref{sec:epsconj} we will recall the formulation of the epsilon constant
conjectures and related known results.
To apply an algorithm of Breuning and the first named  author
for the proof of this conjecture (see \cite{BleyBreu:ExactAlgo})
we will present heuristics to represent local extensions
using global number fields in \autoref{sec:globalrep}.
Thereafter, \autoref{sec:algoproof} gives an overview
of the algorithm and addresses details
and problems which either
needed more efficient solutions or which occurred
during the implementation of the algorithm.
In \autoref{sec:results} we finally
summarize all theoretical results that restrict the problem
to the verification of the local epsilon constant
conjecture for finitely many local extensions of $\QQ_p$.
These problems have then been solved by a computer
and we give some details on the
computations and their results.
Altogether, this will complete the proof of \autoref{thm:loceps}
and its corollaries.
\bigskip

\noindent
\textbf{Notation:} 
For a (local or global) number field $L$ we write $\calO_L$ for its ring of integers.
If $L$ is a local number field with prime ideal $\frP$, we write $U_L$ for the units $(\calO_L)^\times$
and $U_L^{(n)}$ for the $n$-units $1+\frP^n$.

For a group $G$ and a $G$-module $A$ we write
$H^n(G,A)$ for the cohomology group in degree $n$ as defined in \cite[I§2]{NSW:00}
and we will use the inhomogeneous description using
$n$-cochains $C^n(G,A):=\mathrm{Map}(G^n,A)$ and $C^0(G,A):=A$. As usual, we write $\hat{H}^n(G,A)$
for the Tate cohomology groups.

Let $L/K$ denote a local Galois extension with Galois group $G$.
We will also use the notation of class formations of \cite[XI§2]{Serre}
and let $u_{L/K}$ denote the local fundamental class of $L/K$ as defined in
\cite[XIII§3f.]{Serre}, i.e.\ the element which is mapped to $\frac{1}{[L:K]}+\ZZ$
by the local invariant isomorphism $\inv_{L/K}: \TateH^2(G,L^\times)\simeqarrow \frac{1}{[L:K]}\ZZ / \ZZ$.%
\footnote{Similar definitions can be found in 
\cite[(3.1.3), (7.1.4)]{NSW:00}.}

\medskip

\noindent
\textbf{Acknowledgement:}  It is a pleasure for us to thank the referees for their very careful reading
of the manuscript and many valuable comments. In particular, the formulation of our corollaries is motivated
by a suggestion of one of the referees.

\section{An efficient algorithm to compute the local fundamental class}
\label{sec:lfc}
Throughout this section $L/K$ will denote a Galois extension
of $\QQ_p$ with Galois group $G=\Gal(L/K)$.
Our goal is to find the local fundamental class represented 
as a cocycle in $\TateH^2(G, L^\times)$.

A direct method to compute the image of the local fundamental class
under $\TateH^2(G, L^\times) \rightarrow \TateH^2(G, L^\times/ U_L^{(k)})$
for any $k\ge 0$ has been
described in \cite[§2.4]{BleyBreu:ExactAlgo}.
Let $N$ be an unramified extension of $K$ with cyclic Galois group $C$
and of degree $[N:K]=[L:K]$ and let $\Gamma$ denote the Galois group
of $LN/K$.
Then there is a commutative diagram
\vskip-1em
\begin{equation*}
\begin{tikzpicture}[default]
\matrix (m) [mtrx,sizeL]
{
                           &[1em]  \TateH^2(C, N^\times) \\
\TateH^2(G, L^\times)      & \TateH^2(\Gamma, (LN)^\times) \\
\TateH^2(G, L^\times/U_L^{(k)}) & \TateH^2(\Gamma, (LN)^\times/U_{LN}^{(k)}) \\
};
\path[->,font=\scriptsize]
(m-2-1) edge[right hook->] node[ar]{$\inf$} (m-2-2)
(m-3-1) edge[right hook->] node[ar]{$\inf$} (m-3-2)
(m-1-2) edge[right hook->] node[r]{$\inf$} (m-2-2)
(m-2-1) edge (m-3-1)
(m-2-2) edge (m-3-2)
;
\end{tikzpicture}
\end{equation*}
in which all inflation maps are injective, either by \cite[Lem.\ 2.5]{BleyBreu:ExactAlgo} or
\cite[VII, \S 6, Prop.~5]{Serre} combined with Hilbert's Theorem 90.
Based on this diagram, the authors describe an algorithm which consists of the following steps:
\begin{enumerate}
\item Find the fundamental class in $\TateH^2(C, N^\times)$.
\item Compute the image under the composition
\[
\TateH^2(C, N^\times)\stackrel{\inf}{\hooklongrightarrow} \TateH^2(\Gamma, (LN)^\times) \rightarrow \TateH^2(\Gamma, (LN)^\times/U_{LN}^{(k)}).
\]
\item Find the preimage under the map
\[
\TateH^2(G, L^\times/U_L^{(k)}) \stackrel\inf\hooklongrightarrow \TateH^2(\Gamma, (LN)^\times/U_{LN}^{(k)}).
\]
\end{enumerate}
If $\varphi$ denotes the Frobenius automorphism in $C=\langle \varphi\rangle$
and $\pi\in K$ is a uniformizing element, the fundamental class in $\TateH^2(C, N^\times)$
is given by (see~\cite[\S 30, Sec.~4 and \S 31, Sec.~4]{Lorenz})
\[
\gamma(\varphi^i,\varphi^j) = 
\begin{cases}
1 & \text{if } i+j< [N:K], \\
\pi & \text{if } i+j\ge [N:K]
\end{cases}
\]
Since the groups $(LN)^\times/U_{LN}^{(k)}$ and $L^\times/U_L^{(k)}$ are 
finitely generated, one can compute their cohomology groups
using linear algebra \cite{Holt}.
However, this method turns out to be ineffective even for
local fields of small degree.
\bigskip

The basis of a new algorithm to compute the local fundamental class
is the theory of Serre \cite{Serre}
and especially exercise 2 of chapter XIII, §5.
We recall the results of this exercise and show how
to turn it into an efficient algorithm.

Let $E$ be the maximal unramified subextension of $L/K$ and $d:=[E:K]$.
Denote the maximal unramified extension of $K$ by $\mytilde{K}$
and the Frobenius automorphism of $\mytilde{K}/K$ by $\frob$. Then 
$\Gal(\mytilde{K}/K)=\overline{\langle \frob \rangle}$
and $\Gal(\mytilde{K}/E)=\overline{\langle\frob^d\rangle}$. 
We set $\mytilde{L} := \mytilde{K}L$ and note that $\mytilde{L}$ is the maximal unramified
extension of $L$. We always identify 
$\Gal(\mytilde{L}/L)$ with $\Gal(\mytilde{K}/E)$ by restriction.

The Galois group of $\mytilde L/K$ is
given by
\[
\Gal(\mytilde L/K) = \{ (\tau, \sigma) \in \Gal(\mytilde K/K) \times G \,\big|\, \sigma|_E=\tau|_E\}.
\]
We consider $L_{nr}:=\mytilde K \otimes_K L$,
for which we have the following representation:
\begin{lemma}
(i) The map
$L_{nr} = \mytilde K \otimes_K L \rightarrow \prod_{i=0}^{d-1} \mytilde L$
defined by sending elements 
$a \otimes b$ to $\left( a b, \frob(a)b,\ldots, \frob^{d-1}(a)b \right)$
is an isomorphism.

(ii) The Galois action of $\langle\frob\rangle\times G$
 on elements $y=(y_0,y_1,\ldots,y_{d-1})\in \prod_{i=0}^{d-1} \mytilde L$ induced by
this isomorphism is uniquely described by
\begin{align*}
(\frob, 1)(y) &= (y_1,y_2,\ldots,\frob^d(y_0)),\\
(\frob^j, \sigma)(y) &= (\hat\sigma(y_0),\hat\sigma(y_1),\ldots,\hat\sigma(y_{d-1})), \text{if } 
\sigma|_E = \frob^j |_E.
\end{align*}
Here $\hat{\sigma} \in \Gal(\mytilde{L}/K)$ is the unique element such that 
$\hat{\sigma}|_{\wtilde K} = \frob^j \text{ and } \hat\sigma|_L = \sigma$.

\end{lemma}
\begin{proof}
Direct computation, cf.~\cite[XIII §5, Ex.\ 2]{Serre}.
\end{proof}

\begin{remark}
For arbitrary $(\frob^s, \sigma) \in \langle\frob\rangle\times G$ one chooses $j \in \ZZ$
such that $\sigma|_E = \frob^j|_E$. Then $(\frob^s, \sigma) =
(\frob^{s-j}, 1) (\frob^j, \sigma)$ and and the action of each of the factors is given by the lemma.
Explicitly, there is a unique $\hat\sigma \in \Gal(\mytilde{L}/ K)$
such that $\what\sigma|_L = \sigma$, $\what\sigma|_{\tilde K} = \frob^j$ and
$(\frob^s, \sigma)$ acts as $(\frob^{s-j}, 1) \what\sigma$
(with $\what\sigma$ acting diagonally).
\end{remark}

Let $\what L$ be the completion of the maximal unramified extension $\mytilde L$ of $L$.

\begin{lemma}
\label{lem:Neukirch}
For every $c\in U_{\what L}$ 
there exists $x\in U_{\what L}$ such that $x^{\frob^d-1}=c$.
\end{lemma}
\begin{proof}
This is \cite[V, Lem.\ 2.1]{Neu:AlgZThEn} or \cite[XIII, Prop.\ 15]{Serre}
applied to the totally ramified extension
$L/E$ with $\frob^d$ generating $\Gal(\mytilde K/E)$.
Since this will be an essential part of the algorithm,
we sketch the constructive proof of~\cite{Neu:AlgZThEn}.

Denote the residue class field of $\what L$ by $\kappa$,
the cardinality of the residue class field of $E$ by $q$
and let $\phi=\frob^d$. Let $\pi$ be a uniformizing element of $L$.

Since $\kappa$ is algebraically closed, one finds a solution
to $x^\phi=x^q=xc$ in $\kappa$ and one can write
$c=x_1^{\phi-1}a_1$ with $x_1\in U_{\what L}$ and $a_1\in U_{\what L}^{(1)}$.
Similarly, one finds $x_2\in U_{\what L}^{(1)}$ and $a_2\in U_{\what L}^{(2)}$
such that $a_1=x_2^{\phi-1}a_2$. 
Indeed, if we set $a_1 := 1+b_1\pi, \ x_2 := 1+y_2\pi$, then we need to solve 
\[
a_1 x_2^{1-\phi} \equiv 1 - (y_2^\phi - y_2 -b_1)\pi\  \equiv 1 \ (\mathrm{mod } \pi^2),
\]
i.e., we must solve the equation $y_2^q - y_2 - b_1 = 0$ in $\kappa$.

Proceeding this way one has
\[
c= (x_1x_2\cdots x_n)^{\phi-1} a_n,\quad 
x_1\in U_{\what L},\;x_i\in U_{\what L}^{(i-1)},\; a_n\in U_{\what L}^{(n)}
\]
and passing to the limit solves the equation in $U_{\what L}$.
\end{proof}

This fact can be generalized to our case. Let $\what{L}_{nr}$ be the 
completion of $L_{nr}$, so $\what{L}_{nr} \simeq \prod_{i=0}^{d-1} \what L$, and
$w:\what{L}_{nr}\rightarrow\ZZ$ the sum of the valuations.

\begin{lemma}
\label{lem:NeukirchLnr}
For every $c\in \what{L}_{nr}^\times$ with $w(c)=0$
there exists $x\in\what{L}_{nr}^\times$ such that $x^{\frob-1}=c$.
\end{lemma}
\begin{proof}
If $c=(c_0,\ldots c_{d-1})\in\prod_{i=0}^{d-1} \what L^\times$ and $w(c)=0$,
then $\prod_{i=0}^{d-1} c_i \in \what L^\times$ has valuation 0 and there
exists $y\in \what L^\times$ for which $y^{\frob^d-1}=\prod c_i$ by
\autoref{lem:Neukirch}.
Then the element $x=(y,y c_0,y c_0 c_1,\ldots, y c_0\cdots c_{d-2})$
satisfies
\[
x^{\frob-1}
=\frac{(y c_0,y c_0 c_1,\ldots, y c_0\cdots c_{d-2},\frob^d(y))}{(y,y c_0,y c_0 c_1,\ldots, y c_0\cdots c_{d-2})}
=(c_0,c_1,\ldots,c_{d-1})=c
\]
since $\frob^d(y)=y \prod_{i=0}^{d-1} c_i$.
\end{proof}

We prepare our main result by the following lemma.
\begin{lemma}\label{lem:ses-lemma}
\begin{enumerate}[(i)]
\item $\ker(w)=\{y^{\frob-1}\mid y\in \what{L}_{nr}^\times\}$,
\item $\ker(\frob-1)=L^\times$, $L^\times$ being diagonally embedded in $\what{L}_{nr}^\times$, and
\item $\what L_{nr}^\times$ is a cohomologically trivial $G$-module.
\end{enumerate}
\end{lemma}
\begin{proof}
This is \cite[XIII §5, Ex.\ 2(a)]{Serre}. For a detailed proof see \cite[Lemma~2.13]{MyPhD}.
\end{proof}

We denote $V:=\ker(w)$ and from \autoref{lem:ses-lemma} we get the exact sequences
\begin{alignat}{5}
\label{eq:VLZ}
0 &\longrightarrow \;V &&\longrightarrow \,\what L_{nr}^\times
\stackrel{w}\loongrightarrow\ZZ &&\longrightarrow 0 \\
\label{eq:LLV}
\text{and}\quad
0 &\longrightarrow L^\times &&\longrightarrow \,\what L_{nr}^\times
\stackrel{\frob-1}\loongrightarrow V && \longrightarrow 0.
\end{alignat}
Since $\what{L}_{nr}^\times$ is cohomologically trivial,
the connecting homomorphisms of their long exact cohomology sequences
provide isomorphisms
$\delta_1: \TateH^0(G,\ZZ)\rightarrow \TateH^1(G,V)$,
$\delta_2: \TateH^1(G,V)\rightarrow \TateH^2(G,L^\times)$
and we consider the composition
\begin{equation}
\label{eq:Phi-connecting-hom}
\Phi_{L/K}: \TateH^0(G,\ZZ)\simeqlongarrow \TateH^2(G,L^\times).
\end{equation}
Its inverse $\Phi_{L/K}^{-1}$ induces an isomorphism
\[
\overline\inv_{L/K}: \TateH^2(G,L^\times)
\simeq \TateH^0(G,\ZZ) 
\stackrel{\!\cdot\frac{1}{[L:K]}}{\looongrightarrow} \textstyle\frac{1}{[L:K]}\ZZ / \ZZ.
\]

\begin{proposition}
\label{prop:inverse}
\begin{enumerate}[(i)]
\item The map $\overline\inv$ is an invariant map in the sense of \cite[XI, \S 2]{Serre}. Therefore 
the elements $\overline{u}_{L/K}:=\Phi_{L/K}(1+[L:K]\ZZ)$
are fundamental classes with respect to the class formation associated to $\overline{\inv}$.
\item The element $\overline{u}_{L/K}$ is the inverse of the local fundamental class $u_{L/K}$.
\end{enumerate}
\end{proposition}
\begin{proof}
This is \cite[XIII §5, Ex.\ 2(c) and (d)]{Serre}. For a detailed proof we refer the reader 
to  \cite[Prop.~2.14]{MyPhD}. Since parts of the proof will
be turned into an algorithm (see Remark \ref{remark lfc}) we recall some of the details.

Part (i) can be proved by verifying the axioms of a class formation.

In order to prove (ii) it suffices to show $\overline{u}_{N/K} = u_{N/K}^{-1}$ where $N/K$ denotes
the unramified extension of $K$ of degree $[L:K]$. Indeed, by the axioms of a class formation
we have
\[
\inv_{L/K} = \inv_{LN/K} \circ \mathrm{inf}_{L/K}^{LN/K}, \quad
\inv_{N/K} = \inv_{LN/K} \circ \mathrm{inf}_{N/K}^{LN/K},
\]
and the same identities with $\inv$ replaced by $\overline\inv$.
It follows that $\mathrm{inf}_{L/K}^{LN/K}(u_{L/K}) = \mathrm{inf}_{N/K}^{LN/K}(u_{N/K})$ and
$\mathrm{inf}_{L/K}^{LN/K}(\overline{u}_{L/K}) = \mathrm{inf}_{N/K}^{LN/K}(\overline{u}_{N/K})$.
Since $\mathrm{inf}_{L/K}^{LN/K}$ is injective we deduce from $\overline{u}_{N/K} = u_{N/K}^{-1}$
the desired equality  $\overline{u}_{L/K} = u_{L/K}^{-1}$.

For the unramified case one can
make a direct computation of $\Phi_{L/K}(1+[L:K]\ZZ)$
by applying the connecting homomorphisms $\delta_1$ and $\delta_2$
as follows.
For $\delta_1$ we consider the commutative diagram
\begin{equation}
\label{diag:delta1}
\begin{tikzpicture}[default]
\matrix (m) [mtrx,sizeL]
{
  &[-1em]    & \what L_{nr}^\times        & \!\ZZ &[-1em] \\[-2.2em]
  &          & \rotatebox{90}{$=$}        & \rotatebox{90}{$=$} \\[-2.2em]
0 & C^0(G,V) & C^0(G,\what L_{nr}^\times) & C^0(G,\ZZ) & 0 \\
0 & C^1(G,V) & C^1(G,\what L_{nr}^\times) & C^1(G,\ZZ) & 0 \\
};
\path[->,font=\scriptsize]
(m-3-1) edge (m-3-2)
(m-3-2) edge (m-3-3)
(m-3-3) edge node[arr]{$w$} (m-3-4)
(m-3-4) edge (m-3-5)
(m-4-1) edge (m-4-2)
(m-4-2) edge (m-4-3)
(m-4-3) edge node[arr]{$w^*$} (m-4-4)
(m-4-4) edge (m-4-5)
(m-3-2) edge (m-4-2)
(m-3-3) edge node[r]{$\partial_1$} (m-4-3)
(m-3-4) edge (m-4-4)
;
\end{tikzpicture}
\end{equation}
which is induced by the exact sequence \eqref{eq:VLZ}, and where 
$w^*$ is the map on the group of cochains induced by $w$.
If $\pi$ is any uniformizing element of $\what L^\times$,
the element $a=(1,\ldots,1,\pi)\in \what L_{nr}^\times=C^0(G,\what L_{nr}^\times)$
is a preimage of $1$ via $w$.
Applying $\partial_1$ yields
$\alpha\in C^1(G,\what L_{nr}^\times)$, which is defined by
\begin{equation*}
\alpha(\sigma):= 
\frac{\sigma(a)}{a} = 
\begin{cases}
\Big(1,\ldots,1, \frac{\hat\sigma(\pi)}{\pi}\Big), & \text{if } \hat\sigma|_{E}=1 \\[.5em]
\Big(1,\ldots,1,\hat\sigma(\pi),\underbrace{1,\ldots, 1, \textstyle\frac{1}{\pi}\!\!}_{j\text{ components}}\Big), &
\text{if } \hat\sigma|_{E}=\varphi^{-j}, 1\le j\le d-1
\end{cases}
\end{equation*}
The commutativity of the diagram then implies $\alpha\in C^1(G,V)$.

For the connecting homomorphism $\delta_2$ we consider the commutative diagram
\begin{equation}
\label{diag:delta2}
\begin{tikzpicture}[default]
\matrix (m) [mtrx,sizeL]
{
0 &[-1em] C^1(G,L^\times) & C^1(G,\what L_{nr}^\times) & C^1(G,V) &[-1em] 0 \\
0 & C^2(G,L^\times) & C^2(G,\what L_{nr}^\times) & C^2(G,V) & 0 \\
};
\path[->,font=\scriptsize]
(m-1-1) edge (m-1-2)
(m-1-2) edge (m-1-3)
(m-1-3) edge node[ar]{$\varphi-1$} (m-1-4)
(m-1-4) edge (m-1-5)
(m-2-1) edge (m-2-2)
(m-2-2) edge (m-2-3)
(m-2-3) edge (m-2-4)
(m-2-4) edge (m-2-5)
(m-1-2) edge (m-2-2)
(m-1-3) edge node[r]{$\partial_2$} (m-2-3)
(m-1-4) edge (m-2-4)
;
\end{tikzpicture}
\end{equation}
which arises from the exact sequence \eqref{eq:LLV}.
To find a preimage of $\alpha$ via $\frob-1$,
we need elements in $\what L_{nr}^\times$ which are
mapped to $\frac{\sigma(a)}{a}$ by $\varphi-1$.
By \autoref{lem:NeukirchLnr} these preimages are given by
\begin{equation}
\label{eq:defBeta}
\beta(\sigma) := \begin{cases}
\left(u_\sigma,\ldots,u_\sigma\right) & \text{if } \hat\sigma|_{E}=1 \\
(
u_\sigma,\ldots,u_\sigma,
\underbrace{
u_\sigma \hat\sigma(\pi),\ldots,u_\sigma \hat\sigma(\pi)
}_{j \text{ components}}
) & \text{if } \hat\sigma|_{E}=\frob^{-j}, 1\le j\le d-1
\end{cases}
\end{equation}
where $u_\sigma$ solves $u_\sigma^{\varphi^d-1}=\frac{\hat\sigma(\pi)}{\pi}$.
The commutativity of the diagram again implies that the cocycle
\begin{equation}
\label{eq:defGamma}
\gamma(\sigma,\tau):=(\partial_2 \beta)(\sigma, \tau) = \frac{\sigma(\beta(\tau)) \beta(\sigma)}{\beta(\sigma\tau)}
\end{equation}
has values in $L^\times$ and we obtain $\bar u_{L/K}=\Phi_{L/K}(1+[L:K]\ZZ)=\gamma\in \TateH^2(G,L^\times)$.

If $L/K$ is unramified, one can choose $\pi$ to be a uniformizing element of $K$
and set $\sigma=\frob^i$, $\tau=\frob^j$.
Then $\frac{\hat\sigma(\pi)}{\pi}=1$ for all $\hat\sigma\in \Gal(\mytilde L/K)$
and every $u_\sigma\in L^\times$ solves $u_\sigma^{\frob^n-1}=\frac{\hat\sigma(\pi)}{\pi}$.
If one chooses $u_\sigma=\frac1\pi$ for $\sigma\neq id$ and $u_{id}=1$,
one can easily check that
$\overline{u}_{L/K}(\sigma,\tau)=1$, if $i+j<d$ and 
$\overline{u}_{L/K}(\sigma,\tau)=\pi^{-1}$ otherwise.
Hence, $\overline{u}_{L/K}$ is the inverse of the local fundamental class,
cf.~\cite[§2.4]{BleyBreu:ExactAlgo}.
A detailed proof can be found in the second author's dissertation~\cite{MyPhD}.
\end{proof}



\begin{remark}\label{remark lfc}
From the construction above one directly obtains an algorithm.
The uniformizing element $\pi$ of $\what L^\times$ in the proof above
can be chosen to be a uniformizing element of $L$.
Then approximations to the elements $u_\sigma$ can be computed by successively
applying the constructive steps of the proof of
\autoref{lem:Neukirch}.
This involves solving equations in the algebraically closed
residue class field of $\what L^\times$.
However, we cannot do computations in $\what L^\times$ directly,
but rather work in an appropriate subfield,
starting with $L$. Whenever we cannot solve one
of these equations in the residue class field of $L$,
we generate an appropriate algebraic extension
and work there from then on.
In worst case, this means that we have to generate
an algebraic extension in every step.
And, hence, the extensions
involved in the computations often get very large.
\end{remark}

To avoid this problem we proceed as follows. Let $\pi_K$ and $\pi_L$ be uniformizing elements of $K$ and $L$,
$e$ the ramification degree and $d$ the inertia degree of $L/K$.
Let $N$ be the unramified extension of $K$ of degree $[L:K]$. Then $F := LN$ is the
unramified extension of $L$ of degree $e$. We set $L_{nr}:=\prod_d F$
and let $E$ be the maximal unramified extension of $K$ in $L$
with Frobenius automorphism $\varphi$.
In the algorithm below, 
we construct a special uniformizing element $\pi$
in $F$ 
such that $N_{F/L}(\frac{\hat\sigma(\pi)}{\pi})=1$.
One can then prove that the elements $u_\sigma$ can be
constructed in~$F$.

\noindent\parbox{\textwidth}{
\begin{algorithm}[Local fundamental class]
\rm\hspace*{1cm}
\label{alg:lfc}
\Input{An extension $L/K$ over $\QQ_p$ with Galois group $G$ and a precision $k\in\NN$.}
\Output{The local fundamental class $u_{L/K}\in C^2(G, L^\times)$
up to the finite precision $k$, i.e.\ its image in $\TateH^2(G, L^\times/U_L^{(k)})$.}
\begin{enumAlgo}
\item Solve the norm equation $N_{F/L}(v)\equiv u \;{\rm mod}\, U_L^{(k+2)}$ with $u=\pi_K \pi_L^{-e}\in U_L$ and $v\in U_F / U_F^{(k+2)}$.
Define $\pi=v\pi_L$.
\item For each $\sigma\in G$, let $\hat\sigma\in\Gal(F/K)$ be the  automorphism which is uniquely
determined by $\hat\sigma|_L=\sigma$ and $(\hat\sigma|_N)^{-1}=\varphi^j$ with $0\le j\le d-1$.
Then compute $u_\sigma\in U_F$ such that
$u_\sigma^{\varphi^d-1}=\frac{\hat\sigma(\pi)}{\pi}\mod U_F^{(k+2)}$.
\item Define $\beta\in C^1(G,L_{nr}^\times)$ and $\gamma\in C^2(G, L^\times)$ by
\eqref{eq:defBeta} and \eqref{eq:defGamma}.
\end{enumAlgo}
\noindent \textit{Return:} {$\gamma^{-1}$.}
\vspace{-.5em}
\end{algorithm}
}
\medskip

Note that the choice of $j$ in this algorithm corresponds to the choice of $j$ in the proof of \autoref{prop:inverse},
see in particular equation \eqref{eq:defBeta}.

\begin{proof}[Proof of correctness]
Step 1: Since $u$ has valuation 0 and $F/L$ is unramified, there exists an
element $v\in U_F$ such that its norm is equal to $u$.
Then $\pi$ is a uniformizing element of $F$ and has norm
$N_{F/L}(\pi)=u \pi_L^e = \pi_K$.

Step 2: The elements $\frac{\hat\sigma(\pi)}{\pi}$ have norm
\begin{equation*}
N_{F/L}\!\left(\frac{\hat\sigma(\pi)}{\pi}\right)
= \frac{1}{\pi_K} \prod_{i=1}^e \varphi^{di}\Big(\hat\sigma(\pi)\Big)
= \frac{1}{\pi_K} \,\hat\sigma\Big(\prod_{i=1}^e \varphi^{di}(\pi)\Big)
= 1.
\end{equation*}
Let $H=\Gal(F/L)$.
Since $\TateH^{-1}(H, U_F)= {}_{N_H}U_F/I_H U_F=1$ for the unramified extension $F/L$, 
there exists $x\in U_F$ with $x^{\varphi-1}=\frac{\hat\sigma(\pi)}{\pi}$.
By successively applying the steps in the constructive proof
of \cite[V, Lem.\ 2.1]{Neu:AlgZThEn} (see \autoref{lem:Neukirch})
one can construct an element $x\in U_F$ with $x^{\varphi-1}\equiv\frac{\hat\sigma(\pi)}{\pi}\mod U_F^{(k+2)}$.

Step 3: The computation in the proof of \autoref{prop:inverse}
shows that the cocycle $\gamma$ from \eqref{eq:defGamma} represents
the inverse of the local fundamental class.

If we compute the elements $u_\sigma$ modulo $U_F^{(k+2)}$,
we also know the images of $\beta$ to the same precision.
To compute $\gamma^{-1}$ we divide by $\sigma(\beta(\tau))$ and $\beta(\sigma)$
and each of these operations can reduce the precision at most by one
because all elements in (\ref{eq:defBeta}) have at most valuation $1$.
The other operations involved in $\partial_2$ (addition, multiplication
and application of $\sigma$) do not reduce the precision.
Hence, we know the images of $\gamma$ modulo $U_L^{(k)}$.
\end{proof}

This algorithm has been implemented in \textsc{Magma} \cite{Magma} and
its source code is bundled with the second author's dissertation
\cite{MyPhD}.
For a small example where the Galois group is $G=S_3$,
this algorithm computes the local fundamental class within a few seconds
whereas the direct linear algebra method took more than an hour.

The implementation of this more efficient algorithm made several interesting 
applications possible.
In the second author's dissertation \cite{MyPhD}, \autoref{alg:lfc} is used
in algorithms for computations in Brauer groups of (global) number field extensions
and for the computation of global fundamental classes.
In addition, the algorithm was also applied in a completely different context:
based on the Shafarevic-Weil theorem,
Greve used the algorithm in his dissertation \cite{Greve}
to compute Galois groups of local extensions.

\section{Epsilon constant conjectures}
\label{sec:epsconj}
We recall the statements of the global and local epsilon constant conjectures
of \cite{BleyBur:Equiv} and \cite{Breu:LocEps}
and some important related results.
These conjectures are formulated as equations in relative $K$-groups for group rings.

Let $R$ be  an integral domain, $E$ an extension of $\Quot(R)$ and $G$ a finite group.
For a ring $A$ we write $K_0(A)$ for the Grothendieck group of finitely generated
projective $A$-modules
and $K_1(A)$ for the abelianization of the infinite general linear group $\Gl(A)$.
Then there is an exact sequence
\begin{equation}
K_1(R[G]) \rightarrow K_1(E[G]) \stackrel{\partial^1_{R[G],E}}{\vlongrightarrow{35pt}} K_0(R[G],E)
\rightarrow K_0(R[G]) \rightarrow K_0(E[G])
\end{equation}
with the relative algebraic $K$-group $K_0(R[G],E)$ defined in
terms of generators and relations as in \cite[p.\ 215]{Swan}.
An overview of the relevant results concerning these $K$-groups is given in \cite{Breu:PhD}.
We write $\Zent(E[G])$ for the center of $E[G]$ and we
will use the reduced norm map $\nr: K_1(E[G]) \rightarrow \Zent(E[G])^\times$,
which is injective in our cases,
and the map $\widehat\partial^1_{R[G],E}:=\partial^1_{R[G],E}\circ \nr^{-1}$
from $\im(\nr)$ to $K_0(R[G],E)$.

The two cases we are interested in are the following.
For $R=\ZZ_p$ and $E$ an extension of $\QQ_p$ the reduced norm is an isomorphism
(e.g.\ see \cite[Prop.\ 2.2]{Breu:PhD})
and we obtain a map ${\widehat\partial}^1_{G,E}:=\widehat\partial^1_{\ZZ_p[G], E}
=\partial^1_{\ZZ_p[G],E} \circ \nr^{-1}$
from $\Zent(E[G])^\times$ to $K_0(\ZZ_p[G],E)$.

For $R=\ZZ, E=\RR$ the reduced norm map is not surjective
but the decomposition 
\begin{equation}
\label{eq:decomp}
K_0(\ZZ[G],\QQ) \simeq \coprod_p K_0(\ZZ_p[G], \QQ_p),
\end{equation}
and the Weak Approximation Theorem
still allow us to define a canonical map ${\widehat\partial}^1_{G,\RR}$
from $\Zent(\RR[G])^\times$ to $K_0(\ZZ[G],\RR)$, such that
$ {\widehat\partial}^1_{G,\RR} \circ \nr = {\partial}^1_{G,\RR}$, 
cf. \cite[§3.1]{BleyBur:Equiv} or~\cite[Lem.~2.2]{BreBur:07:LeadingTerms}.


\subsection{The global epsilon constant conjecture}
The global epsilon constant conjecture is formulated in the
relative $K$-group $K_0(\ZZ[G],\RR)$.
For a Galois extension $L/K$ of number fields 
it describes a relation between the epsilon factors arising in  the functional
equation of Artin $L$-functions and algebraic invariants
related to $L/K$. We briefly sketch its formulation which is due to Burns and the first named author 
and refer to \cite{BleyBur:Equiv} for more details.

The completed Artin $L$-function $\Lambda(L/K,\chi,s)$  
satisfies the functional equation
\begin{equation}
\Lambda(L/K,\chi,s) = \varepsilon(L/K,\chi,s)\; \Lambda(L/K,\bar\chi,1-s)
\end{equation}
with an epsilon factor
$\varepsilon(L/K,\chi,s):=W(\chi) A(\chi)^{\frac{1}{2}-s}$
and $W(\chi), A(\chi)$ as defined in \cite[Chp.\ I, (5.22)]{Froe:GalMod}.
The equivariant epsilon function is defined by
$\varepsilon(L/K, s) := \left(\varepsilon(L/K,\chi,s)\right)_{\chi\in\Irr(G)}$
and its value $\epsilon_{L/K}:=\varepsilon(L/K, 0)\in \Zent(\RR[G])^\times$
is called the equivariant global epsilon constant.
We define a corresponding element in the relative $K$-group $K_0(\ZZ[G],\RR)$ by
$\scrE_{L/K} := \widehat\partial^1_{G,\RR}(\epsilon_{L/K})$ and
also refer to it as the equivariant global epsilon constant.


Let $S$ be a finite set of non-archimedean places of $K$, including all non-archimedean places which ramify in $L$.
For each $v\in S$ with $v|p$ we fix a place $w$ of $L$ above $v$ and choose
a full projective $\ZZ_p[G_w]$-sublattice $\scrL_w$ of $\calO_{L_w}$
upon which the $v$-adic exponential map is well-defined (and hence injective).
For each place $w$ which does not lie above some $v\in S$ we set $\scrL_w=\calO_{L_w}$
and we define $\scrL\subseteq \calO_L$ by its $p$-adic completions
\begin{equation*}
\scrL_p = \prod_{v|p} \scrL_w \otimes_{\ZZ_p[G_w]} \ZZ_p[G]
\subseteq L_p := L\otimes_\QQ \QQ_p.
\end{equation*}
For each finite place $w$ of $L$ we also write $w : L^\times \longrightarrow \ZZ$ for the standard valuation of $L$
normalized such that $w(L^\times) = \ZZ$.
We let $\Sigma(L)$ denote the set of all embeddings of $L$ into $\CC$ and set  
$H_L:=\prod_{\sigma\in\Sigma(L)} \ZZ$.
We define the $G$-equivariant discriminant by $\delta_{L/K}(\scrL):=\![\scrL, \pi_L, H_L]\in K_0(\ZZ[G],\RR)\!$
where $\pi_L$ is induced by
$\rho_L:L\otimes_\QQ \CC\rightarrow H_L\otimes_\ZZ\CC$, $l\otimes z \mapsto (\sigma(l)z)_{\sigma\in\Sigma(L)}$
as in \cite[§3.2]{BleyBur:Equiv}.

Let $X\subseteq \calO_{L_w}^\times$ be any  cohomologically trivial $\ZZ[G_w]$-submodule of finite index, 
e.g.\ $X\!=\exp_v(\scrL_w)$.
Then $\TateH^2(G_w, L_w^\times)\simeq \TateH^2(G_w, L_w^\times/X)$ and
by \cite[Th.\ 2.2.10]{NSW:00} 
there is an isomorphism $\TateH^2(G_w, L_w^\times/ X)\simeq\Ext^2_{G_w}(\ZZ, L_w^\times/ X)$.
For a cocycle $\gamma\in \TateH^2(G_w, L_w^\times/ X)$ one can apply the
construction of \cite[p.\ 115]{NSW:00} to obtain a 2-extension
\[
0\rightarrow L_w^\times/ X\rightarrow C(\gamma)\rightarrow\ZZ[G_w]\rightarrow\ZZ\rightarrow 0
\]
representing $\gamma$ in $\Ext^2_{G_w}(\ZZ, L_w^\times/ X)$.

If $\gamma$ represents the local fundamental class then the complex 
\[
K^\bullet_w(X) := \big[ C(\gamma)\rightarrow\ZZ[G_w]\big]
\]
is perfect. Here the  modules are placed in degrees $0$ and $1$.
We write $E_w(X)$ for its refined Euler characteristic in $K_0(\ZZ[G_w],\QQ)$
where the trivialization 
\[
H^0(K^\bullet_w(X)) \otimes_\ZZ \QQ \simeq 
L_w^\times/X\otimes_\ZZ \QQ \stackrel\simeq\longrightarrow \QQ \simeq H^1(K^\bullet_w(X)) \otimes_\ZZ \QQ
\]
is induced by the valuation map $w : L_w^\times \rightarrow \ZZ$.
For the general construction of refined Euler characteristics we refer the reader to
\cite[§2]{Burns:Whitehead}.
For the construction in our special case see \cite[§3.3]{BleyBur:Equiv}, in particular,
a triple representing $E_w(X)$ in $K_0(\ZZ[G_w],\QQ)$ is given in
\cite[Lem.\ 3.7]{BleyBur:Equiv}.

Furthermore, let $m_w\in \Zent(\QQ[G_w])^\times$ be the element defined
in \cite[§\,4.1]{BleyBur:Equiv} which we also call the \emph{correction term.}
It is defined as follows.
For a subgroup $H\subseteq G$ and $x\in \Zent(\QQ[H])$
we let ${}^*x\in \Zent(\QQ[H])^\times$ denote the invertible element
which on the \emph{Wedderburn decomposition} $\Zent(\QQ[H])=\prod_{i=1}^r F_i$
with suitable extensions $F_i/\QQ$
is given by
${}^*x=({}^*x_i)_{i=1\ldots r}$
with 
${}^*x_i=1$ if $x_i=0$ and ${}^*x_i=x_i$ otherwise.
If $\varphi_w$ denotes a lift of the Frobenius automorphism in $G_w/I_w$,
then the correction term is defined by
\begin{equation}
\label{eq:def-correction}
m_w=\frac{ {}^*(|G_w/I_w|e_{G_w})\cdot {}^*((1-\frob_w{\Nr\!v}^{-1}) e_{I_w})
 }{  {}^*( (1-\frob_w^{-1})e_{I_w} ) }
\in \Zent(\QQ[G_w])^\times.
\end{equation}

Finally, we define elements
\begin{align*}
I_G(v,\scrL)
&:={\rm i}_{G_w}^G\big( {\widehat\partial}^1_{G_w, \RR}(m_w) - E_w(\exp_v(\scrL_w)) \big) \\
\text{and}\quad T\Omega^{\rm loc}(L/K,1) &:=
\scrE_{L/K} - \delta_{L/K}(\scrL) -
\textstyle\sum_{v\in S} I_G(v,\scrL)
\end{align*}
in $K_0(\ZZ[G],\RR)$. One can show that $T\Omega^{\rm loc}(L/K,1)$
is independent of the choices of $S$ and $\scrL$ (cf.\ \cite[Rem.\ 4.2]{BleyBur:Equiv}).
By \cite[Prop.\ 3.4]{BleyBur:Equiv}) we have   $T\Omega^{\rm loc}(L/K,1) \in K_0(\ZZ[G],\QQ)$ 
and we can state the conjecture as follows.

\begin{conjecture}[Global epsilon constant conjecture]
\label{conj:glob}
For every finite Galois extension $L/K$ of number fields
the element
$T\Omega^{\rm loc}(L/K,1)$ is zero in $K_0(\ZZ[G],\QQ)$.
We denote this conjecture by $\EPS(L/K)$.
\end{conjecture}

This conjecture has been proved for tamely ramified extensions (\cite[Cor.~7.7]{BleyBur:Equiv}),
for abelian extensions $L/\QQ$ (see the proof of our \autoref{cor:globeps}),
for all $S_3$-extensions $L/\QQ$ (\cite{Breu:Dihedral}), and finally, for certain infinite families of  
dihedral and quaternion extensions (\cite{Breu:LocEps}).
Moreover, the global conjecture $\EPS(L/K)$ is known to be valid modulo the subgroup $K_0(\ZZ[G],\QQ)_\tor$,
i.e.,\ $T\Omega^{\rm loc}(L/K,1)\in K_0(\ZZ[G],\QQ)_\tor$ (see \cite[Cor.~6.3]{BleyBur:Equiv}).

We write $\EPS_p(L/K)$ for the projection of the conjecture
onto $K_0(\ZZ_p[G], \QQ_p)$ via the decomposition \eqref{eq:decomp}.
We immediately obtain

\begin{corollary}
The global conjecture $\EPS(L/K)\!$ is valid if and only if
its $p$-part $\EPS_p(L/K)$ is valid for all primes $p$.
\end{corollary}

\subsection{The local epsilon constant conjecture}
We will now describe a related conjecture  for local Galois extensions
$L_w/K_v$ over $\QQ_p$, which was formulated by Breuning in \cite{Breu:LocEps},
and we will see how it refines the global conjectures
$\EPS(L/K)$ and $\EPS_p(L/K)$.
The equivariant global epsilon function of $L/K$ can be written as
a product of equivariant local epsilon functions related
to its completions $L_w/K_v$.
Their value at zero is called the equivariant local
epsilon constant and the local conjecture describes
it in terms of algebraic elements of the extension $L_w/K_v$.
Here we refer to \cite{Breu:LocEps} 
for details.

Let $\CC_p$ denote the completion of an algebraic closure $\QQ_p^c$ of $\QQ_p$.
Every character $\chi$ of $G_w=\Gal(L_w/K_v)$ 
can be viewed as a character of $\Gal(\QQ_p^c / K_v)$.
The local Galois Gauss sum from \cite[Chp.~II, §\,4]{Martinet} associated with the induced
character ${\rm i}_{K_v}^{\QQ_p}\chi$   of $\Gal(\QQ_p^c / \QQ_p)$ will be denoted 
by $\tau_{L_w/K_v}(\chi)\in \CC$ and we set
\[
\tau_{L_w/K_v} := \big( \tau_{L_w/K_v}(\chi) \big)_{\chi\in\Irr_\CC(G_w)}
\in \Zent(\CC[G_w])^\times.
\]
The choice of an embedding $\iota\!:\CC\!\rightarrow \CC_p$ induces a map
$\Zent(\CC[G_w])^\times\!\!\rightarrow\! \Zent(\CC_p[G_w])^\times$
and we obtain the \emph{equivariant local epsilon constant}
\[
T_{L_w/K_v} := \widehat\partial^1_{G_w, \CC_p}(\iota(\tau_{L_w/K_v}))\in K_0(\ZZ_p[G_w], \CC_p).
\]

As in the global case one chooses a full projective $\ZZ_p[G_w]$-sublattice $\scrL_w$
of $\calO_{L_w}$ upon which the exponential function is well-defined.
Similarly one defines the \emph{equivariant local discriminant}
in $K_0(\ZZ_p[G_w],\CC_p)$ by
$\delta_{L_w/K_v}(\scrL_w)=[\scrL_w, \rho_{L_w}, H_{L_w}]$,
where $H_{L_w}=\bigoplus_{\sigma\in\Sigma(L_w)}\ZZ_p$
and $\rho_{L_w}$ is the isomorphism
$\scrL_w \otimes_{\ZZ_p} \CC_p \rightarrow H_{L_w}\otimes_{\ZZ_p} \CC_p$,
$l\otimes z\mapsto (\sigma(l) z)_{\sigma\in\Sigma(L_w)}.$
Here $\Sigma(L_w)$ denotes the set of embeddings $L_w\hookrightarrow \CC_p$.
There is an explicit description of $\delta_{L_w/K_v}(\scrL_w)$  
in \cite[§\,4.2.5]{BleyBreu:ExactAlgo} which we will
recall in \autoref{sec:lec-terms}.

\index{Euler characteristic}
We write $E_w(\exp_v(\scrL_w))_p$ for the projection of the Euler characteristic $E_w(\exp_v(\scrL_w))$ onto
$K_0(\ZZ_p[G_w], \QQ_p)$ by \eqref{eq:decomp}.
The difference $C_{L_w/K_v} := E_w(\exp_v(\scrL_w))_p-\delta_{L_w/K_v}(\scrL_w)$
is independent of $\scrL_w$ by \cite[Prop.~2.6]{Breu:LocEps}
and is called the \emph{cohomological term} of $L_w/K_v$.

To state the local conjecture we also need 
the \emph{unramified term} $U_{L_w/K_v}$.
It is the unique element in $K_0(\ZZ_p[G_w], \CC_p)$ which satisfies the 
Galois invariance property  \cite[Prop.~2.12 b)]{Breu:LocEps}
and is mapped to zero by the scalar extension map 
$K_0(\ZZ_p[G_w], \QQ_p)\rightarrow K_0(\calO^t_p[G_w], \CC_p)$
where $\calO^t_p$ is the ring of integers of the maximal tamely ramified
extension of $\QQ_p$ in $\CC_p$.
The proof of its existence in \cite[Prop.~2.12]{Breu:LocEps}
includes an explicit description of $U_{L_w/K_v}$,
which we will also recall in \autoref{sec:lec-terms}.

\begin{conjecture}[Local epsilon constant conjecture]
For every Galois extension $L_w/K_v$ 
of local fields over $\QQ_p$ 
the element
\[
R_{L_w/K_v} := T_{L_w/K_v} + C_{L_w/K_v}
+ U_{L_w/K_v} - \widehat\partial^1_{G_w, \CC_p}(m_w)
\]
is zero in $K_0(\ZZ_p[G_w],\CC_p)$.
We denote this conjecture by $\EPSloc(L_w/K_v)$.
\end{conjecture}

This conjecture has been proved in  \cite{Breu:LocEps} for tamely ramified extensions,
for abelian extensions $M/\QQ_p$ with $p\neq2$, for all $S_3$-extensions of $\QQ_p$ ($p$ arbitrary),
and for certain  other special cases.
Actually some of the results on the global conjecture were obtained
by the local conjecture which 
can be regarded as a refinement of the $p$-part
of the global conjecture.

\begin{theorem}[Local-global principle]
\label{thm:localglobal}
One has the equality
\[
T\Omega^{\rm loc}(L/K,1)_p = \sum_{v|p} {\rm i}_{G_w}^G (R_{L_w/K_v})
\]
in $K_0(\ZZ_p[G], \QQ_p)$
and one can deduce:
\begin{enumerate}[(i)]
\item $\EPSloc(E/F)$ for all $E/F/\QQ_p$ $\Rightarrow$ $\EPS_p(L/K)$ for all $L/K/\QQ$,

\item if $p\neq 2$: $\EPS_p(L/K)$ for all $L/K/\QQ$ $\Rightarrow$ $\EPSloc(E/F)$ for all $E/F/\QQ_p$, and

\item for fixed $L/K/\QQ$ and $p$: $\EPSloc(L_w/K_v)$ for all $w|v|p$ $\Rightarrow$ $\EPS_p(L/K)$.
\end{enumerate}
\end{theorem}
\begin{proof}
\cite[Thm.\ 4.1, Cor.\ 4.2 and Thm.\ 4.3]{Breu:LocEps}.
\end{proof}

As a consequence, for $p\neq2$, parts (i) and (ii) imply the equivalence of the local
conjecture for extensions of $\QQ_p$ and the $p$-part of the global conjecture.

\subsection{An algorithm}
We recall the functorial properties of the global and local epsilon constant conjectures.

\begin{proposition}[Functorial property]
\label{prop:epsconj-functorial}
For a Galois extension $L/K$ of number fields with intermediate
field $F/K$
and a local Galois extension $M/N$ over $\QQ_p$
with intermediate field $E/N$ one has:\pagebreak[2]
\begin{enumerate}[(i)]
\item $\EPS(L/K) \Rightarrow \EPS(L/F)$
and $\EPS(L/K) \Rightarrow \EPS(F/K)$ if $F/K$ is Galois.
\item $\EPSloc(M/N) \Rightarrow \EPSloc(M/E)$
and $\EPSloc(M/N) \Rightarrow \EPSloc(E/N)$ if $E/N$ is Galois.
\end{enumerate}
\end{proposition}
\begin{proof}
\cite[Thm.~6.1]{BleyBur:Equiv} and
\cite[Prop.~3.3]{Breu:LocEps}.
\end{proof}

The functorial properties together with the known results mentioned so far imply the following corollary.
\begin{corollary}
\label{cor:locglob}
Let $n\in\NN$ be a fixed integer.
Then the local epsilon constant conjecture $\EPSloc(M/\QQ_p)$
for all extensions $M/\QQ_p$ of degree $[M:\QQ_p]\le n$
with $p\le n$ implies the
global epsilon constant conjecture $\EPS(F/K)$ for all
Galois extensions $F/K$ where $F$ can be embeded into
a Galois extension $L/\QQ$ of degree $[L:\QQ]\le n$.
\end{corollary}
\begin{proof}
All extensions below are assumed to be Galois. We conclude
\vspace{0.5cm}

\begin{tabular}{rlll}
              & $\EPSloc(M/\QQ_p)$ & $\forall [M:\QQ_p]\le n, p\le n$ \\[.3em]
$\Rightarrow\!\!\!\!$ & $\EPSloc(M/\QQ_p)$ & $\forall [M:\QQ_p]\le n, \forall p$ 
 & (since   $\EPSloc(M/\QQ_p)$ is  \\[.3em]
&&& valid for tame extensions) \\[.3em]
$\Rightarrow\!\!\!\!$ & $\EPS_p(L/\QQ)$    & $\forall [L:\QQ]\le n, \forall p$ 
 & (by \autoref{thm:localglobal} (iii)) \\[.3em]
$\Rightarrow\!\!\!\!$ & $\EPS(L/\QQ)$      & $\forall [L:\QQ]\le n$
 & (by decomposition \eqref{eq:decomp}) \\[.3em]
$\Rightarrow\!\!\!\!$ & $\EPS(F/K)$       & $\forall F\subseteq L, [L:\QQ]\le n$
 & (by \autoref{prop:epsconj-functorial})
\end{tabular}\\
\hspace*{2cm}
\end{proof}

It is well-known that for fixed $p$ and $n$ there are only finitely
many Galois extensions $M/\QQ_p$ with degree $[M:\QQ_p]= n$.
So the local conjecture for finitely many extensions implies the
global conjecture for an infinite number of extensions.
And these finitely many local extensions can be handled algorithmically:

\begin{enumerate}
\item For a fixed positive  integer $n$, compute for all $p \le n $
all local Galois extensions of $\QQ_p$ of degree $\le n$.
This can be done using an algorithm due to Pauli and Roblot \cite{PauliRoblot}
which performs well enough up to degree 15.
However, we were not able to compute all local extensions
of degree 16 of $\QQ_2$.

\item For every local extension $M/\QQ_p$, 
find a global Galois extension $L/K$ of number fields with places $w|v$,
such that $L_w=M$, $K_v=\QQ_p$ and $[L:K]=[M:\QQ_p]$.
Such an extensions $L/K$ is called a \emph{global representation} for $M/\QQ_p$
and is needed to do exact computations in step (3).

\item Apply the algorithm of Breuning and the first named  author \cite{BleyBreu:ExactAlgo} to
prove or disprove  the local epsilon constant conjecture for these extensions.
\end{enumerate}

In the next section we will discuss how step (2) can be handled.
Afterwards, we recall the algorithm of \cite{BleyBreu:ExactAlgo}, and finally, we 
present our algorithmic results and their consequences.


\section{Global representations of local Galois extensions}
\label{sec:globalrep}
To do exact computations for a fixed Galois extension $M/\QQ_p$
in the algorithm of Breuning and the first named  author,
we will need a global Galois extension $L/K$ of number fields
with corresponding primes $\frP|\frp$ for which $K_\frp=\QQ_p$ and $L_\frP=M$.
Such an extension $L/K$ will be called global representation for $M/\QQ_p$
and is denoted by $(L,\frP)/(K,\frp)$.

The proof of the existence of such a global representation
involves the Galois closure of a number field
\cite[Lem.~2.1 and~2.2]{BleyBreu:ExactAlgo},
but for computational reasons we need a representation which has
small degree over $\QQ$, or even better, with $K=\QQ$.

Henniart shows in \cite{Henniart} that a global representation
$L/K$ for the local extension $M/\QQ_p$ exists with $K=\QQ$
if $p\neq 2$. And if $p=2$, there exists a global representation
with $K$ quadratic over $\QQ$.
Unfortunately, it is not clear how to find these small representations
algorithmically.
We therefore present some heuristics.

\subsection{Search database of Klüners and Malle}
The database of Klüners and Malle \cite{KM} contains
polynomials generating Galois extensions of $\QQ$
for all subgroups $G$ of permutation groups $S_n$
up to degree $n=15$.
In particular, the database contains polynomials
for all Galois groups of order $n\le 15$.
Among those one will often find a polynomial
generating a global representation for $M/\QQ_p$,
if $[M:\QQ_p]\le 15$.

\subsection{Parametric polynomials}
Here we consider polynomials $f\in K(t_1,\ldots, t_n)[x]$
with indeterminates $t_i$ over a field $K$.
Such a polynomial $f$ is said to be \emph{parametric} for a given group $G$,
if the splitting field $L$ of $f$ is a Galois extension of $K(t_1,\ldots, t_n)$ with group isomorphic to $G$
and, moreover, if for every Galois extension $N/K$ with $\Gal(N/K) \simeq G$ there
exist parameters $\alpha_1, \ldots, \alpha_n \in K$ such that the splitting field
of $f(\alpha_1, \ldots, \alpha_n)[x] \in K[x]$ is isomorphic to $N$.
Since $K$ is countable, one can systematically enumerate those
polynomials $f$ and one will eventually find a polynomial whose splitting field
is a global representation for $M$ (provided it exists).
In our applications, we could find such a polynomial $f$ by
randomly testing different values for the indeterminates $t_i$.

The book \cite{genpols} by Jensen et.\ al.\ contains parametric polynomials
(or methods to construct them) for a lot of groups.
In particular, it contains polynomials for all
non-abelian groups of order $\le 15$, except for the
generalized quaternion group $Q_{12}$ of order $12$.
However, there do not exist parametric polynomials for all groups.
The smallest group for which the non-existence is proved
is the cyclic group of order $8$ \cite[§\,2.6]{genpols}.

\subsection{Class field theory}
As a last heuristic, we will use class field theory to construct
abelian extensions with prescribed ramification.%
\footnote{Thanks to Jürgen Klüners for suggesting the application of this method.}
A discussion of class field theoretic algorithms implemented in \textsc{Magma}
is given by Fieker in~\cite{Fieker:CFT}. For the general theory we refer the reader
to \cite[Ch.~VI]{Neu:AlgZThEn}.

If $L/K$ denotes an abelian extension of number fields with conductor $\frf$, then
$\frp|\frf$ if and only if $\frp$ is ramified in $L/K$ and, moreover,
$\frp^2|\frf$ if and only if $\frp$ is wildly ramified in $L/K$,
cf.\ \cite[§~2.4, p.\,44]{Fieker:CFT}.

One can therefore possibly find abelian extensions of $K$ with prescribed
ramification at certain places by choosing an appropriate modulus $\frf$,
constructing the corresponding ray class field, and
computing suitable subfields of the requested degree.

\subsection{Global representations for extensions up to degree 15}
\label{sec:glrep-results}

Let $M/\QQ_p$ be a Galois extension of local fields with group $G$.
For the computation of the unramified characteristic (see (\ref{eq:unramified-term})) 
we will also have to consider
the unramified extension $N_f$ of $\QQ_p$ of degree $f=\exp(G^{\rm ab})$,
where $f$ denotes the exponent of the abelianization $G^{\rm ab}$ of $G$.
Note that for Algorithm \ref{alg:lfc} we do not need global representations of
the fields used in the algorithm. 

Since the local conjecture is known to be valid for tamely ramified extensions
and abelian extensions of $\QQ_p$, $p\neq2$, it suffices to  discuss the
performance of the heuristic methods in the following cases:
\begin{enumerate}[\quad(a)]
\item wildly ramified extensions $M$ of $\QQ_p$ with non-abelian Galois group $G$ for all primes $p$,
\item wildly ramified extensions $M$ of $\QQ_2$, with abelian Galois group $G$, and
\item unramified extensions of $\QQ_p$ of degree $f=\exp(G^{\rm ab})$ for all primes $p$.
\end{enumerate}
In all of these cases we restrict to extensions of degree $\le15$ since
for degree 16 we cannot compute all extensions of $\QQ_2$.
The hypothesis of wild ramification implies that we only have to consider
primes $p=2,3,5$ and $7$.
The primes $11$ and $13$ are not considered because they can only occur
(up to degree $\le15$) in abelian extensions of degree $11$ and $13$,
for which both the local and global epsilon conjecture is known to be true.

\subsubsection{Case (a)}
First consider extensions with \emph{non-abelian} Galois group.
For most of the non-abelian wildly-ramified local extensions
we found polynomials of the appropriate degree in the database \cite{KM}
generating a global representation.
In fact, there were just three $D_4$--extensions of $\QQ_2$
and three $D_7$--extensions of $\QQ_7$ 
not being represented by any polynomial (of degree 8 or 14 respectively)
in this database.

By \cite[Cor.~2.2.8]{genpols} every $D_4$--extension of $\QQ$
is the splitting field of a polynomial
$f(x)=x^4-2st x^2 + s^2t(t-1)\in \QQ[x]$
with suitable $s,t\in\QQ$.
Experimenting with small integers $s$ and $t$ and computing the
splitting field of $f$ quickly provides global
representations for all $D_4$--extensions of $\QQ_2$.

Finally, we used class field theory to construct
global Galois representations for the three non-isomorphic
$D_7$--extensions of $\QQ_7$:
by taking quadratic extensions $K$ of $\QQ$ which are 
non-split at $p=7$ and computing all
$C_7$--extensions of $K$ which are subfields of the ray class field 
$K^\frm$, $\frm=49\calO_K$,
one finds $D_7$--extensions where $p=7$ is ramified 
with ramification index $7$ or $14$ and where $p$ does not split.
Experimenting with different fields $K$ as above one
finds global Galois representations for all three
$D_7$--extensions of $\QQ_7$.

This completes the construction of global representations
for all non-abelian wildly ramified local extensions 
of $\QQ_p$, $p=2,3,5,7$, up to degree 15.

\subsubsection{Case (b)}
Using the database \cite{KM}
we can again find polynomials for all abelian extensions
over $\QQ_2$ of degree $\le 7$.
For extensions of higher degree, the heuristics were not
as successful.
But to obtain a \emph{global} result up to degree 15,
it is sufficient to consider abelian extension
of $\QQ_2$ of degree $\le 7$ (see the proof of \autoref{cor:globeps}).

\subsubsection{Case (c)}
For each of the pairs $(L/\QQ, p)$ with Galois group $G$
constructed in cases (a) and (b),
\autoref{alg:locepsconst} also needs
a extension $N$ of $\QQ$ which is unramified and non-split at $p$
and is of degree $f=\exp(G^{\rm ab})$.

For non-abelian extensions of degree $\le 15$
the maximum degree of $N$ can easily be determined to be $f=4$.
And in the abelian case, we need unramified extensions of degree $\le 7$.

Most of these unramified extensions can be constructed
as a subfield of a cyclotomic field $\QQ(\zeta_n)$
generated by an $n$-th root of unity $\zeta_n$.
In the other cases one finds global representations
using the database \cite{KM}.
\bigskip

A complete list of polynomials which were found using these
heuristics is contained in the second named author's dissertation \cite{MyPhD}.

\section{Algorithmic proof of the local epsilon constant conjecture}
\label{sec:algoproof}
We briefly recall the algorithm described 
by Breuning and the first named  author in \cite[§4.2]{BleyBreu:ExactAlgo}.
There the authors explain in
detail how each of the terms in the local conjecture can be computed
and how this results in an algorithmic proof of the local conjecture
for a given local Galois extension $L_w/K_v$.
Since by the functorial properties of the local conjecture
one has $\EPSloc(L_w/\QQ_p)\Rightarrow \EPSloc(L_w/K_v)$,
we will only consider extensions $L_w/\QQ_p$.

For the rest of this section,
fix the Galois extensions $L/K$ and $N/K$ and a prime $\frp$ of $K$
as the input of the algorithm. We assume that $L/K$ (resp. $N/K$) is
a global representation of $L_w/\QQ_p$ (resp. the unramified extension of $\QQ_p$ of
degree $f := \exp(G^\mathrm{ab})$).
For simplicity, the unique prime ideal above $\frp$ 
in the fields $L$, $N$, or any subextension of $L/K$
will also be denoted by $\frp$.
If it is necessary to avoid confusion, we will write $\frp_K$, $\frp_L$
and $\frp_N$.
Furthermore, we will identify the ideals $\frp_L|\frp_K$
with places $w|v$ of $L$ and $K$, respectively, such that $L_w=L_\frp$
and $K_v=K_\frp$. We write $e_{w|v} = e(L_w/\QQ_p)$ for the ramification index.
Recall that for a finite place $w$ of $L$ we write $w : L^\times \longrightarrow \ZZ$ for
the normalized valuation associated with $w$ (or $\frp_L$).

We will first recall the complete algorithm of \cite{BleyBreu:ExactAlgo}
and then explain each step. In step 2 we will construct a big number field $E$ which, among
other things, is a splitting field for $G$. Hence the Wedderburn decomposition of $E[G]$ induces a
canonical isomorphism $\Zent(E[G])^\times \simeq \prod_{\chi \in \Irr(G)} E^\times$.

\pagebreak[3]

\begin{algorithm}[Proof of the local epsilon constant conjecture]
\rm\hspace*{1cm}
\label{alg:locepsconst}
\Input{An extension $(L,\frP)/(K,\frp)$ with $K_\frp=\QQ_p$
in which $L/K$ is Galois with group $G$
and a Galois extension $N/K$ of degree $\exp(G^{\rm ab})$
in which $\frp$ is non-split and unramified.}
\Output{True if $\EPSloc(L_\frP/\QQ_p)$ was successfully checked.}
\begin{enumAlgo}
\vspace{-.5em}
\algpart{Construction of the coefficient field}
\item Compute all characters $\chi$ of $G$ and use
\index{Brauer induction}
Brauer induction to find an integer $t$
such that the Galois Gauss sums can be computed in $\QQ(\zeta_m,\zeta_{p^t})$,
$m=\exp(G^{})$ (cf. \cite[Rem.~2.7]{BleyBreu:ExactAlgo}).
\item Construct the composite field $E$ of $L, N$ and  $\QQ(\zeta_m,\zeta_{p^t})$ and 
fix a complex embedding $\iota:E\hookrightarrow\CC$
and a prime ideal $\frQ$ of $E$ above $p$.

\vspace{-.3em}
\algpart{Computation of cohomological term}
\vspace{-.2em}
\item Let $\theta\in L$ be a generator of a normal basis of $L/K$
with $w(\theta)> \frac{e(L_w/\QQ_p)}{p-1}$,
define $\scrL=\ZZ_p[G]\theta\in \calO_{L_w}$
and compute $k$ such that $(\frP \calO_{L_w})^k \subseteq \scrL$ (cf. \cite[Sec.~4.2.3]{BleyBreu:ExactAlgo}).
\item Compute a cocycle representing the local fundamental class
up to precision $k$ in $\TateH^2(G,L_w^\times/U_{L_w}^{(k)})$ 
and its projection 
onto $\TateH^2(G,L_w^\times/\exp_v(\scrL_w))$ (cf. \autoref{alg:lfc}).
\item Construct a complex representing this cocycle by \cite[p.~115]{NSW:00}
and compute the Euler characteristic $E_w(\exp_v(\scrL_w))\in K_0(\ZZ[G], \QQ)$
as in \cite[§4.2.4]{BleyBreu:ExactAlgo}.

\vspace{-.4em}
\algpart{Computation of the terms in  $\prod_\chi E^\times$}
\vspace{-.2em}
\item Compute the correction term 
$m_{L_\frP/\QQ_p}\!=m_{w}\!\in \Zent(\QQ[G])^\times\!\subseteq \Zent(E[G])^\times\!\simeq \prod_\chi E^\times\!$
defined in \eqref{eq:def-correction}.
\item Compute the element
$d_{L_\frP/\QQ_p}\in L[G]^\times\subseteq E[G]^\times$
of \eqref{eq:equiv-discriminant} below, such that $\nr(d_{L_\frP/\QQ_p})$ represents 
the equivariant discriminant $\delta_{L_\frP/\QQ_p}(\scrL)$.
\item Compute the element $u_{L_\frP/\QQ_p}\in N[G]^\times\subseteq E[G]^\times$
using \eqref{eq:unramified-term} below, such that $\nr(u_{L_\frP/\QQ_p})$ represents the unramified term
$U_{L_\frP/\QQ_p}$.
\item Use the
canonical homomorphism $E[G]^\times\rightarrow K_1(E[G])$,
\index{reduced norm}%
the reduced norm map $\nr: K_1(E[G])\rightarrow \Zent(E[G])$
\index{Wedderburn decomposition}%
and Wedderburn decomposition of $\Zent(E[G])$
to represent these two terms in $\prod_\chi \!E^\times$.
\item Compute the equivariant epsilon constant
$\tau_{L_\frP/\QQ_p}\in\prod_\chi \QQ(\zeta_{p^t},\zeta_m)^\times\subseteq \prod_\chi E^\times$
via Galois Gauss sums.

\pagebreak[3]
\vspace{-.5em}
\algpart{Computations in relative $K$-groups}
\item Read $E_w(\exp_v(\scrL_w))$ and the tuples from above 
as elements in $K_0(\ZZ_p[G], E_\frQ)$.
\item Compute the sum  $R_{L_\frP/\QQ_p}\in K_0(\ZZ_p[G], E_\frQ)$ of the resulting elements.
\end{enumAlgo}
\Return{True if $R_{L_\frP/\QQ_p}$ is zero, and false otherwise.}
\end{algorithm}
\begin{proof}
\cite[§4.2]{BleyBreu:ExactAlgo}.
\end{proof}
\pagebreak[3]

All steps were explained in detail in \cite{BleyBreu:ExactAlgo}.
However, there were some problems that needed further
improvements to give a practical algorithm.
Firstly, the existence of global representations is due to a
theoretical argument by Henniart in \cite{Henniart}
which we still cannot make explicit.
For the construction of these representations we gave some
heuristics in the previous section which we successfully
applied to extensions of small degree.
Secondly, the computation of local fundamental classes
as presented in \cite[§\,2.4]{BleyBreu:ExactAlgo}
is not very efficient and is significantly improved by
\autoref{alg:lfc}.
And thirdly, Wilson and the first named author \cite{BleyWilson} developed new algorithms 
for computations in the relative algebraic $K$-groups $K_0(\ZZ_p[G], \QQ_p)$.

Below we will discuss each part of the algorithm separately.

\subsection{Constructing the coefficient field}
\label{sec:lec-coefffield}
As explained in \cite[§\,4.2.2]{BleyBreu:ExactAlgo}
we need to construct a global field $E$, in which all the computations
take place.

For the computation of the unramified term, we will need a
cyclic extension $N/K$ which is unramified and non-split at $\frp$.

Another extension involved is $\QQ(\zeta_m, \zeta_{p^t})$,
where $m$ is the exponent of $G^{}$ and $t$ is computed as
described below.
By Brauer's theorem \cite[Sec.~12.3, Theorem 24]{Serre:Rep} the field $\QQ(\zeta_m)$
is a splitting field for all irreducible characters of $G$ and therefore 
contains all character  values.
The root of unity $\zeta_{p^t}$ is used to represent
Galois Gauss sums and the integer $t$ is determined as follows.

For each character $\chi$ of $G$
one computes subgroups $H$, linear characters $\phi$ of $H$,
and coefficients $c_{(H,\phi)}\in\ZZ$
such that $\chi-\chi(1) 1_G=\sum_{(H,\phi)} c_{(H,\phi)} \ind_H^G(\phi-1_H)$.
Such a relation exists by 
Brauer's induction theorem, cf.\
\cite[§\,2.5]{BleyBreu:ExactAlgo}.
If $\frf(\phi)$ denotes the \emph{Artin conductor} of $\phi$ and
$e$ the ramification index of $(L^H)_\frp/\QQ_p$,
then $t$ must satisfy $t\ge v_\frp(\frf(\phi))/e$
for all pairs $(H,\phi)$ and all $\chi$.
This choice of $t$ will allow 
to compute the epsilon constants as elements of $\QQ(\zeta_m, \zeta_{p^t})$,
see also \cite[Rem.~2.7]{BleyBreu:ExactAlgo}.

The composite field of the three fields $L,N$ and $\QQ(\zeta_m, \zeta_{p^t})$
is denoted by $E$, giving the following situation: 
\begin{equation}
\label{diag:fieldsituation}
\begin{tikzpicture}[default]
\matrix (m) [mtrx,flds]
{
  &[-2ex] E \\
\QQ(\zeta_m,\zeta_{p^t})\hspace{-1ex} & L &[.5ex] \hspace{-1ex}N \\
  & \QQ \\
};
\path[-]
(m-3-2) edge (m-2-1)
        edge (m-2-2)
        edge (m-2-3)
(m-1-2) edge (m-2-1)
        edge (m-2-2)
        edge (m-2-3)
;
\end{tikzpicture}
\end{equation}

\pagebreak[3]
We then fix a complex embedding $\iota: E\hookrightarrow \CC$.
The embedding $\iota$ is essential because some of the terms
in the conjecture depend on 
the particular choice of the embedding: for example,
the definition of the standard additive character below,
see also \cite[§\,2.5]{BleyBreu:ExactAlgo}.
So once we
compute an algebraic element representing this value,
we have to 
maintain its embedding into $\CC$.
Since we still try to avoid computations in such a big field $E$,
this implies the following:
whenever we do calculations in a subfield $F\subseteq E$,
we have to choose embeddings $\iota_1: F\hookrightarrow \CC$
and $\iota_2:F\hookrightarrow E$ such that
the diagram
\begin{equation*}
\begin{tikzpicture}[default]
\matrix (m) [mtrx,sizeM]
{
E &[1ex] \CC \\[.5ex]
F \\
};
\path[right hook->]
(m-1-1) edge node[ar]{$\iota$} (m-1-2)
(m-2-1) edge node[left]{$\iota_2$} (m-1-1)
(m-2-1) edge node[below right,shift={(-0.1,0.1)}]{$\iota_1$}(m-1-2)
;
\end{tikzpicture}
\end{equation*}
is commutative, i.e.\ $\iota_1=\iota \circ \iota_2$.

We also fix a prime ideal $\frQ$ of $E$ above $p$
and an embedding $E\hookrightarrow E_\frQ$ such that
$E\hookrightarrow E_\frQ \hookrightarrow \CC_p$
and $E\stackrel\iota\hookrightarrow \CC \hookrightarrow \CC_p$
coincide.
Then all the invariants appearing in the conjecture lie in the
subgroup $K_0(\ZZ_p[G], E_\frQ)$ of $K_0(\ZZ_p[G], \CC_p)$
and they can therefore be represented by tuples in
$\Zent(E_\frQ[G])\simeq \prod_{\chi\in \Irr(G)} E_\frQ^\times$.
In fact, we will see that all these elements are already represented by
elements in $\prod_{\chi\in \Irr(G)} E^\times$
and can therefore be computed globally.

\subsection{Computation of  the cohomological term}
\label{sec:lec-cohomterm}
The lattice $\scrL=\ZZ[G]\theta\subseteq \calO_L$ is computed using
a normal basis element $\theta$ for $L/K$ and 
the integer $k$ for which $\frp^k\subseteq \scrL$
can then be found experimentally. The details are explained in 
\cite[§\,4.2.3]{BleyBreu:ExactAlgo}).

We compute a cocycle
$\gamma \in Z^2(G,L_w^\times/U_{L_w}^{(k)})$
representing the local fundamental class up to precision $k$
using \autoref{alg:lfc} and then the projection of $\gamma$ onto 
$Z^2(G,L_w^\times/\exp_v(\scrL_w))$. Note that the quotient $L_w^f := L_w^\times/\exp_v(\scrL_w)$
can be computed globally, cf. \cite[Rem.~3.6]{Bley:NumEv}.
We can then construct the corresponding complex
$P_w=\big[L_w^f(\gamma) \rightarrow \ZZ[G]\big]$
using the splitting module $L_w^f(\gamma)$
of \cite[Chp.~III, §\,1, p.~115]{NSW:00}
and the Euler characteristic $E_w(\exp_v(\scrL_w))\in K_0(\ZZ[G], \QQ)$
using the explicit construction of \cite[§\,4.2.4]{BleyBreu:ExactAlgo}.

\subsection{Computation of the terms in $\prod_\chi E^\times$}
\label{sec:lec-terms}
The correction term $m_w$ is explicitly defined as a tuple in $\prod_\chi E^\times$
by \eqref{eq:def-correction}.
For the equivariant discriminant and the unramified term
we recall the following formulas from \cite[§§\,4.2.5 and 4.2.7]{BleyBreu:ExactAlgo}:
\begin{align}
\label{eq:equiv-discriminant}
d_{L_w/\QQ_p} &=
\sum_{\tau\in G} \tau(\theta) \tau^{-1}
\in L[G]^\times\subseteq E[G]^\times, \\
\label{eq:unramified-term}
u_{L_w/\QQ_p} &=
\sum_{i=0}^{s-1} \frob_\frp^i(\xi) \sigma^{-i}
\in N[G]^\times\subseteq E[G]^\times.
\end{align}
We explain the notation in the latter equation. The element $\sigma \in G$ 
is a lift of the local norm residue symbol
$(p, F_\frp/K_\frp)\in \Gal(F_\frp/K_\frp)\simeq\Gal(F/K)$
with $F$ being the maximal abelian subextension in $L/K$.
An algorithm to compute local norm residue symbols is
described in \cite[Alg.~3.1]{Klue:Artin}.
If $s$ denotes the order of $\sigma$ and $N_1\subseteq N$
is the subextension of $N/K$ of degree $[N_1:K]=s$,
then we denote the Frobenius automorphism of $N_1/K$
with respect to $\frp$ by $\frob_\frp\in\Gal(N_1/K)$,
and finally, $\xi\in\calO_{N_1}$ is an integral normal basis element  
for the extension $(N_1)_\frp/K_\frp$.


By \cite[Sec.~4.2.5 and 4.2.7]{BleyBreu:ExactAlgo} the equivariant discriminant and
the unramified term are represented by $\nr(d_{L_w/\QQ_p})$ and $\nr(u_{L_w/\QQ_p})$ 
as elements in   $\Zent(E[G])^\times \simeq \prod_{\chi \in \Irr(G)} E^\times$.
Recall that reduced norms can be computed by one of the algorithms described in
\cite[Sec.~3.2]{BleyBreu:ExactAlgo} or \cite[Sec.~3.3]{BleyWilson}.



%

\bigskip

The equivariant epsilon constant $\tau_{L_\frp/\QQ_p}$
is computed in $\prod_\chi E^\times$ by \emph{local Galois Gauss sums}
as follows.

\label{sec:epsilon-constant}
For each $\chi$,
we have already computed subgroups $H$ of $G$, linear characters $\phi$ of $H$,
and coefficients $c_{(H,\phi)}\in\ZZ$ such that 
$\chi-\chi(1) 1_G=\sum_{(H,\phi)} c_{(H,\phi)} \ind_H^G(\phi-1_H)$.
Since Galois Gauss sums are additive, inductive in degree $0$, and equal to $1$
for the trivial character, we obtain 
\[
\tau(L_\frp/\QQ_p,\chi) = \prod_{(H,\phi)}
\tau\big((L^{\ker(\phi)})_\frp/(L^H)_\frp, \phi\big)^{c_{(H,\phi)}}
\in \QQ(\zeta_m,\zeta_{p^t})\subseteq E^\times.
\]
For the completions at $\frp$  of the abelian extension $M=L^{\ker(\phi)}$ over $N=L^H$,
Galois Gauss sums are given by the formula
\[
\tau(M_\frp/N_\frp,\phi) =
\sum_{x}
\phi\!\left(\left( \frac{x}{c}, M_\frp/N_\frp \right)\right)
\psi_{N_\frp}\!\!\left(\frac{x}{c}\right)
\in \QQ(\zeta_m,\zeta_{p^t})\subseteq  E^\times
\]
where $x$ runs through a system of representatives of
$\calO_{N_\frp}^\times / U_{N_\frp}^{(s)}\simeq (\calO_N/\frp^s)^\times$,
$s$ is the valuation $v_\frp(\frf(\phi))$
of the \emph{Artin conductor} $\frf(\phi)$ of $\phi$,
$c\in N$ generates the ideal $\frf(\phi)\calD_{N_\frp}$,
where $\calD_{N_\frp}$ denotes the \emph{different} of the extension $N_\frp/\QQ_p$,
and $\psi_{N_\frp}$ is the \emph{standard additive character} of $N_\frp$.

The above formulas allow the construction of the equivariant
epsilon constant as a tuple
$\tau_{L_\frp/\QQ_p}=\big(  \tau(L_\frp/\QQ_p, \chi) \big)_\chi\in \prod_\chi E^\times$.
For details see \cite[§\,2.5]{BleyBreu:ExactAlgo}.

\subsubsection{Computations in relative $K$-groups}
\label{sec:lec-relKgroup}
In the following we have to combine the computations
of the previous steps to find $R_{L_\frp/\QQ_p}$ and show
that its sum represents zero in $K_0(\ZZ_p[G],E_\frQ)$.
In \cite{BleyWilson} Wilson and the first named  author describe the relative $K$-group
as an abstract group. Using their methods it will be
clear how to read elements of the form 
$\widehat\partial^1_{G_w,\QQ_p}(x)$ for $x\in\prod_\chi E^\times$
and triples $[A,\theta,B]$ in the group
$K_0(\ZZ_p[G],E_\frQ)$.

We will recall the algorithms of 
\cite{BleyWilson} and
--- since they are not yet implemented in full generality ---
we will discuss a simple modification for special extensions $F$ of $\QQ$ which are
totally split at the fixed prime $p$.

First we introduce some more notation:
Let $K$ be a number field and $G$ a finite group. Let $\Irr_K(G) = \{\chi_1, \ldots, \chi_r\}$ 
denote a set of
orbit representatives of $\Irr(G)$ modulo the action of $\Gal(K^c/K)$.
Then the  \emph{Wedderburn decomposition} of $K[G]$ induces a decomposition
of its center $C:=\Zent(K[G])$ into character fields $K_i :=K(\chi_i), i = 1, \ldots, r$, so that we have
$C=\bigoplus_{i=1}^r K_i$.

Choose a maximal $\calO_K$-order $\calM$ of $K[G]$
containing $\calO_K[G]$
and a two-sided ideal $\frf$ of $\calM$ 
which is contained in $\calO_K[G]$ (e.g.\ $\frf=|G|\,\calM$)
and define $\frg:=\calO_C \cap \frf$.
Then the decomposition of $C$ similarly splits $\calM$ into $\bigoplus_{i=1}^r \calM_i$
and the ideals $\frf$ and $\frg$ into ideals $\frf_i$ of $\calM_i$
and $\frg_i$ of $\calO_{K_i}$.
For a prime $\frp$ in $\calO_K$, we further write $C_\frp$ for the localization
$C_\frp=K_\frp \otimes_\QQ C=\bigoplus_{i=1}^r K_\frp \otimes_\QQ K_i
=\bigoplus_{i=1}^r \bigoplus_{\frP|\frp} (K_i)_\frP$,
and $\fra_{i,\frp}$ for the part of an ideal $\fra_i$ of $\calO_{K_i}$ above $\frp$.

The reduced norm map induces a homomorphism
$\mu_\frp: K_1(\calO_{K_\frp}[G]/\frf_\frp)
\rightarrow $\linebreak $\bigoplus_{i=1}^r ( \calO_{K_i} / \frg_{i,\frp} )^\times$
whose cokernel is used 
in the description of
the relative $K$-group $K_0(\calO_{K_\frp}[G], K_\frp)$.
Then the main algorithmic result of Wilson and the first named  author is the following.

\pagebreak[3]
\begin{proposition}
\label{prop:K0Rel-iso}
There are isomorphisms
\[
K_0(\calO_{K_\frp}[G], K_\frp) \stackrel{\bar n}{\longrightarrow} C_\frp^\times / \nr(\calO_{K_\frp}[G]^\times)
\stackrel{\bar \varphi}{\longrightarrow} I(C_\frp) \times \coker(\mu_\frp),
\]
$\bar n$ being the natural isomorphism of \cite[Th.~2.2(ii)]{BleyWilson} 
and $\bar\varphi$ being induced by
\begin{equation}
\label{eq:varphi}
\begin{array}{cccc}
\varphi: & C_\frp^\times=\displaystyle\bigoplus_{i=1}^r\, (K_i)_\frp & \longrightarrow & I(C_\frp) \times \displaystyle\bigoplus_{i=1}^r\, (\calO_{K_i}/\frg_{i,\frp})^\times \\
& (\nu_1,\ldots,\nu_r) & \longmapsto & \left( \big( \prod_\frP \frP^{v_\frP(\nu_i)} \big)_{\!i}, (\bar\mu_1,\ldots,\bar\mu_r)  \right),
\end{array}
\end{equation}
where $\mu_i:=\nu_i \prod_\frP \pi_{i,\frP}^{-v_\frP(\nu_i)}$
and $\pi_{i,\frP}\in\calO_{K_i}$ are uniformizing elements having valuation 1 at $\frP$ and 
which are congruent to 1 modulo $\frg_{\frP'}$ for all other primes $\frP'$ above $\frp$ in $K_i/K$.
\end{proposition}
\begin{proof}
\cite[Prop.~2.7]{BleyWilson}.
\end{proof}
\pagebreak[3]

Wilson and the first named  author describe an algorithm to compute the
group $I(C_\frp) \times \coker(\mu_\frp)$.
From the definition of $\varphi$, it is clear how a tuple $\nu=(\nu_i)_i$
of elements with values $\nu_i\in K_i$ represents an element in this group.
Furthermore, for every triple $[A,\theta,B]\in K_0(\calO_{K}[G], K)$
with projective $\calO_{K}[G]$-modules $A$ and $B$
and $\theta:A_{K} \simeqarrow B_{K}$, the algorithm of \cite[Sec.~4.1]{BleyWilson}
produces a representative of  $[A,\theta,B]$ in $\Zent(K[G])^\times$. In this way, we can
compute a representative of $E_w(\exp_v(\calL_w))$ in $\Zent(\QQ[G])^\times \subseteq
\Zent(E[G])^\times \simeq \prod_\chi E^\times$.

In theory, this solves the remaining problems for \autoref{alg:locepsconst}. Indeed,
since we have representatives of each of the individual terms in $\prod_\chi E^\times$,
we can represent $R_{L_w/\QQ_p}$ by an element $(a_\chi)_\chi \in \prod_\chi E^\times 
\subseteq \prod_\chi E_\frQ^\times$.  
If $F \subseteq E$ denotes the decomposition field of  $\frQ$ and $\frq := \frQ \cap F$,
then $K_0(\calO_{F_\frq}[G], F_\frq) = K_0(\ZZ_p[G], \QQ_p)$ and 
$K_0(\calO_{F_\frq}[G], E_\frQ) = K_0(\ZZ_p[G], E_\frQ)$. Since $R_{L_w/\QQ_p} \in K_0(\ZZ_p[G], \QQ_p)
\subseteq K_0(\ZZ_p[G], E_\frQ)$ it follows that $\omega(a_\chi) = a_{\omega\circ\chi}$ for all $\chi \in \Irr(G)$
and all $\omega \in \Gal(E/F)$ (see \cite[bottom of page 788]{BleyBreu:ExactAlgo}). In other words,
$a_\chi \in F(\chi)$ for each $\chi \in \Irr(G)$ and $\sigma(a_\chi) = a_{\sigma\circ\chi}$ 
for all $\sigma \in \Gal(F(\chi)/F)$.
So  the natural approach using
the work of \cite{BleyWilson} would be to work with the $\frq$-part of $K_0(\calO_F[G], F)$.
But in practice, unfortunately, this has only been implemented in \textsc{Magma} for $K=\QQ$ and $\frp=p\ZZ$.

The field $F$ is an extension of $\QQ$ which is totally split at $p$.
We obviously have $F_\frq=\QQ_p$ and $K_0(\ZZ_p[G],F_\frq)\simeq K_0(\ZZ_p[G],\QQ_p)$.
If $F$ satisfies certain conditions, this isomorphism of relative $K$-groups
is \emph{canonically} given by isomorphisms on the ideal part $I(C_\frp)$ and the cokernel part $\coker(\mu_\frp)$.

\begin{proposition}
\label{prop:K0Rel-iso-F}
Let $K = \QQ$ and $\frp = p\ZZ$. Let $F/\QQ$ be a Galois extension which is totally split at $p$
and for which $F\cap K_i=\QQ$ for all $i = 1, \ldots,r$.
Let $\frq$ be a fixed prime ideal of $F$ above $p$.
Then the following holds:
\begin{enumerate}[(i)]
\item The center $C'=\Zent(F[G])$ splits into character fields $F_i=FK_i$.
\item For every ideal $\frP$ of $K_i$ there is exactly one prime ideal $\frQ$
in $F_i$ lying above $\frP$ and $\frq$.
\item There are canonical isomorphisms
\[
I(C_\frp)\simeq I(C'_\frq)
\quad\text{and}\quad
\bigoplus_{i=1}^r (\calO_{K_i}/\frg_{i,\frp})^\times
\simeq 
\bigoplus_{i=1}^r (\calO_{F_i}/\frg'_{i,\frq})^\times
\]
where $\frg':= \frg \calO_{C'}$.
\item One can use $\calM' := \calO_F \otimes_\ZZ \calM$ and   $\frf' := \calO_F \otimes_\ZZ \frf$
in order to compute $K_0(\calO_{F_\frq}[G], F_\frq)$ as in \autoref{prop:K0Rel-iso}. Denote the corresponding homomorphisms
in \eqref{eq:varphi} by $\varphi'$ and $\mu'_\frq$. If we use the same uniformizing elements for $\varphi$ and $\varphi'$,
then
\vskip-1em
\begin{equation*}
\begin{tikzpicture}[default]
\matrix (m) [mtrx,sizeL]
{
C_\frp^\times / \nr(\ZZ_p[G]^\times) &[1em]  I(C_\frp) \times \coker(\mu_\frp) \\
{(C'_\frq)}^\times / \nr(\calO_{F_\frq}[G]^\times) &[1em]  I(C'_\frq) \times \coker(\mu'_\frq) \\
};
\path[->,font=\scriptsize]
(m-1-1) edge node[ar]{$\bar\varphi$} (m-1-2)
(m-2-1) edge node[ar]{$\bar\varphi'$} (m-2-2)
(m-1-1) edge node[r]{$\simeq$} (m-2-1)
(m-1-2) edge node[r]{$\simeq$} (m-2-2)
;
\end{tikzpicture}
\end{equation*}
commutes. Here the right hand vertical isomorphism is induced by (iii).
\end{enumerate}
\end{proposition}
\begin{proof}
(i) Since $F$ and $K_i$ are disjoint over $\QQ$, one has $\Irr_\QQ(G)=\Irr_F(G)$ and $F(\chi_i) = F K(\chi_i)$.

(ii) If $\frQ'$ is any prime ideal in $F_i$ above $\frp$ and $\frP'=\frQ'\cap K_i$,
$\frq'=\frQ'\cap F$, then the automorphisms $\tau$ and $\sigma$ for which
$\tau(\frP')=\frP$ and $\sigma(\frq')=\frq$ define an element $\rho=\sigma\times\tau$
in the Galois group of $F_i/\QQ$ and $\frQ=\rho(\frQ')$ is a prime ideal which lies
above both $\frP$ and $\frq$.
Since $F\cap K_i=\QQ$ and $F/\QQ$ is totally split at $\frp$,
the extension $F_i/K_i$ is also totally split at every prime ideal $\frP'$ above $\frp$.
The uniqueness of $\frQ$ therefore follows from degree arguments.

(iii) Let $\frP$ be a prime ideal of $K_i$ and $\frQ$ the prime ideal of $F_i$
which lies above $\frq$ and $\frP$.
Then the valuation $v_\frQ$ of $F_i$ extends the valuation $v_\frP$ of $K_i$
and if we identify each pair $\frP$ and $\frQ$, we get an isomorphism
\[
I(C_\frp)
  =    \prod_{i=1}^r \prod_{\frP|\frp} \frP^\ZZ
\simeq \prod_{i=1}^r \prod_{\frQ|\frq} \frQ^\ZZ
  =    I(C'_\frq).
\]

Since $\frP\subset  K_i$ is totally split in  $F_i$ 
we have isomorphisms $\calO_{K_i} / \frP \simeq \calO_{F_i}/ \frQ$.
The inclusions $\calO_{K_i}\subseteq\calO_{F_i}$ therefore
induce isomorphisms
$(\calO_{K_i}/\frg_{i,\frp})^\times
\simeq(\calO_{F_i}/\frg'_{i,\frq})^\times$.

(iv) Since $\frq | \frp$ is unramified the order $\calM'$ is maximal at $\frq$. This implies the
first part of (iv). The commutativity follows from  straightforward verification.
\end{proof}

\subsection{Further remarks}
1.\ As mentioned before, the algorithms of \cite{BleyWilson} to compute
$K_0(\ZZ_p[G], F_\frq)$ are just implemented for $F=\QQ$.
The extension to $F/\QQ$ described above
will work if $F$ is totally split at $p$, $F/\QQ$ is Galois,
and $F\cap \QQ(\chi)=\QQ$ for all characters $\chi$.
The first condition is always true
since we want to work with the decomposition field $F\subseteq E$ of $\frQ$,
and the latter conditions are valid in all cases we consider in
the computational results below.
\medskip

2.\ 
The computation of the prime ideal $\frQ$ in $E$ is a very hard problem when the
degree of $E$ gets large. 
In the last part of \autoref{alg:locepsconst} we will therefore
try to replace $E$ by a much smaller field $E'$. 
Let $\calI:=\tau_{L_w/\QQ_p} u_{L_w/\QQ_p} / (m_wd_{L_w/\QQ_p}) \in \prod_\chi E^\times$
be the element combining all the invariants except the cohomological term.
So $R_{L_w/K_v}=\widehat\partial^1_{G,E_\frQ}(\calI)+E_w(\exp_v(\scrL_w))_p$
and since $R_{L_w/K_v}$ and $E_w(\exp_v(\scrL_w))_p$ are both elements of $K_0(\ZZ_p[G],\QQ_p)$,
the element $\widehat\partial^1_{G,E_\frQ}(\calI)$ is also in $K_0(\ZZ_p[G],\QQ_p)$.
As in  \cite[bottom of page 788]{BleyBreu:ExactAlgo} we deduce that  $\calI_\chi\in F(\zeta_m)$
where $m=\exp(G)$ and $F=E^{G_\frQ}$ denotes the decomposition field of $\frQ$.

To compute a small field $E'$ without computing the ideal $\frQ$ and its decomposition group
itself we proceed as follows: for every $\chi$ we compute the
minimal polynomial $m_\chi$ of $\calI_\chi$.
Then we compute the composite field $E'$ of the splitting fields of
the polynomials $m_\chi$ and $\QQ(\zeta_m)$.
We note that the splitting fields will always be subfields of $E$.
The computation of these fields is also
a difficult task, but where this approach could take hours, the computation of $\frQ$
did not succeed in several days.

In the end, $E'$ is a subfield of $E$ such that $\calI_\chi,\zeta_m\in E'$.
Compute an ideal $\frq'$ of $E'$ above $p$,
denote the decomposition field of $\frq'$ by $F$,
and compute $\frq=\calO_F\cap \frq'$.
Then it follows from above that $\calI_\chi\in F(\zeta_m)$
and $\calI=\tau_{L_w/\QQ_p} u_{L_w/\QQ_p} / (m_wd_{L_w/\QQ_p}) \in \prod_\chi F(\zeta_m)^\times$.
In our computations, these fields $F$ were at most of degree 4 over $\QQ$
and they were always Galois so that we could apply \autoref{prop:K0Rel-iso-F}
for the remaining computations in $K_0(\ZZ_p[G], F_\frq)=K_0(\ZZ_p[G],\QQ_p)$.

Note that all computations were independent of the choice of the
prime ideal $\frQ$ above $p$ because all invariants were actually computed globally.
The proof of the conjecture will therefore also be independent of
the choice of $\frq'$.

\section{Computational results}
\label{sec:results}
\autoref{alg:locepsconst} has been implemented in \textsc{Magma}
\cite{Magma} and is bundled with the second author's dissertation \cite{MyPhD}.%
It has been tested for various extensions up to degree 20
and the computation time turns out to depend essentially on the degree of the
composite field $E$.

The most complicated extension for which we proved the local
epsilon constant conjecture was an extension of degree 10
of $\QQ_5$ with Galois group $D_5$.
The composite field $E$ then had degree 200 over $\QQ$.
The computation of the epsilon constants,
which needs an embedding $E\hookrightarrow \CC$, already took about
7 hours, but the most time-consuming part (about 6.5 days)
of \autoref{alg:locepsconst} was
the computation of minimal polynomials
and their splitting fields mentioned in the remarks above.
The field $E'$ then just had degree 4 over $\QQ$ making the
remaining computations very fast.
The total time needed to prove the local conjecture in this
case was about 7 days.%
\footnote{All computations were performed with \textsc{Magma}
version 2.15-9 on a
dual core AMD Opteron machine with 1.8 GHz and 16 GB memory.}

Using the representations obtained in \autoref{sec:globalrep}
we can prove \autoref{thm:loceps} algorithmically.

\begin{proof}[Proof of \autoref{thm:loceps}]
Since the local conjecture is valid for abelian extensions of $\QQ_p$, $p\neq2$,
the only primes to consider are $p=2,3,5,7$.
All local extensions for these primes of degree $\le 15$ that
are either non-abelian, or abelian with $p=2$ and of degree $\le 6$
have been considered in \autoref{sec:glrep-results}
and global representations have been found
using the heuristics described in \autoref{sec:globalrep}.
Also global representations for the corresponding unramified extensions
--- which are of degree at most 6 --- could be found using
the database \cite{KM}.

For each of those extensions we then continued with
\autoref{alg:locepsconst} to prove the local
epsilon constant conjecture.
Details of the computations can be found in the author's
dissertation \cite{MyPhD}.
This completes the proof of \autoref{thm:loceps}.
\end{proof}

Using some already known results we can also prove:
\begin{corollary}
\label{cor:loceps}
The local epsilon constant conjecture is valid for Galois extensions
\begin{enumerate}[(a)]
\item $M/\QQ_p$, $p\neq2$, of degree $[M:\QQ_p]\le 15$,
\item $M/\QQ_2$ non-abelian and of degree $[M:\QQ_2]\le 15$,
\item $M/\QQ_2$ of degree $[M:\QQ_2]\le 7$.
\end{enumerate}
\end{corollary}
\begin{proof}
The cases not considered in the theorem above are
extensions of $\QQ_p$, $p\neq2$, which are either tamely
ramified or have abelian Galois group,
and extensions of $\QQ_2$ which are tamely ramified.
These cases have already been proved in \cite{Breu:LocEps}.
Note that for degree 7 there is just one extension of $\QQ_2$
which is also tamely ramified.
\end{proof}

We finally provide the proofs of \autoref{cor:globeps} and \autoref{cor:deg 15}.

\begin{proof}[Proof of \autoref{cor:globeps}]
We recall that $\EPSloc(E/F)$ is true for all tamely ramified Galois extensions and
for all abelian extensions $E /\QQ_p$, $p \ne 2$. Combining these results with 
\autoref{cor:loceps} we deduce  \autoref{cor:globeps} from \autoref{thm:localglobal} (iii).
\end{proof}

\begin{proof}[Proof of \autoref{cor:deg 15}]
If $L/\QQ$ is abelian, the global conjecture is already known to be valid
by combining \cite[Cor.~1.3]{BreBur:Dedekind} with \cite[Thm.~5.2]{BreBur:07:LeadingTerms}.
Note that the compatibility conjecture of Breuning and Burns
stated in \cite[Conj.~5.3]{BreBur:07:LeadingTerms} is equivalent to \autoref{conj:glob}
(which is the conjecture Bley and Burns stated in \cite{BleyBur:Equiv})
by \cite[Rem.~5.4]{BreBur:07:LeadingTerms}.

If $L/\QQ$ is Galois with $[L : \QQ] \le 15$ then $L/\QQ$ does not satisfy Property~$(*)$
if and only if  $2$ is wildly ramified with abelian decomposition group $G_w$ such that
$8 \le |G_w| \le 15$. But in these cases $\Gal(L/\QQ) = G_w$
is abelian and we can use the above mentioned result for the global absolutely abelian case.
\end{proof}

\hyperref[cor:omega2]{Corollaries \ref*{cor:omega2}} \hyperref[cor:relative]{and \ref*{cor:relative}}
finally follow from \autoref{cor:globeps}.
\pagebreak[3]

\bibliographystyle{amsplain}

\end{document}